\documentclass[12pt]{amsart}

\calclayout

\usepackage{amssymb,amsmath,amsthm, amsfonts}
\usepackage{mathtools, enumitem, wasysym}

\usepackage{xcolor}
\usepackage{hyperref}
\hypersetup{colorlinks=false,linkbordercolor=red,linkcolor=green,pdfborderstyle={/S/U/W 2}}

\usepackage{url}

\usepackage{graphicx}
\usepackage{wrapfig,tikz, tikz-cd}
\usetikzlibrary{arrows, arrows, calc, decorations.markings}

\usepackage{caption}
\usepackage{subcaption}

\newtheorem{theorem}{Theorem}[section]
\newtheorem{proposition}[theorem]{Proposition}
\newtheorem{lemma}[theorem]{Lemma}

\theoremstyle{definition}
\newtheorem{definition}[theorem]{Definition}

\theoremstyle{remark}

\numberwithin{equation}{section}

\newcommand{\dd}{\partial}

\newcommand{\cut}{\operatorname{cut}}

\newcommand{\al}{\alpha}
\newcommand{\be}{\beta}

\newcommand{\ZZ}{\mathbb{Z}}
\newcommand{\BB}{\mathfrak{B}}
\newcommand{\MM}{\mathfrak{M}}
\newcommand{\kk}{\textbf{k}}

\begin{document}

\title{Quivers with potentials for cluster varieties associated to braid semigroups.}

\author{Efim Abrikosov}
\address{}
\curraddr{Yale University, Department of Mathematics, 10 Hillhouse Ave, 4th Floor}
\email{efim.abrikosov@yale.edu}
%\thanks{Alexander Goncharov}

%\subjclass[2016]{}
%\date{12 December, 2016}
%\dedicatory{}

\begin{abstract}
Let $C$ be  a simply laced generalized Cartan matrix.  Given an element $b$ 
of the generalized braid semigroup related to $C$, we construct a collection of mutation-equivalent quivers with potentials. A quiver with potential in such a collection corresponds to an expression of $b$ in terms of the standard generators. For two expressions that differ by a braid relation, the corresponding quivers with potentials are related by a mutation.

The main application of this result is a construction of a family of $CY_3$ $A_\infty$-categories associated to elements of the braid semigroup related to $C$. 

In particular, we construct a canonical up to  equivalence  $CY_3$ $A_\infty$-category associated to quotient of any Double Bruhat cell $G^{u,v}/{\rm Ad} H$ in a simply laced reductive Lie group $G$.

We describe the full set of parameters  these categories depend on by defining a 2-dimensional CW-complex and proving that the set of parameters is identified with second cohomology group of this complex.
\end{abstract}

\maketitle

%For pictures------------------------------------------------
%mygreen
\definecolor{mygreen}{RGB}{10,160,50}

\tikzset{->-/.style={decoration={
			markings,
			mark=at position #1 with {\arrow{>[scale=2.5, length = 3, width = 2.5]}}},postaction={decorate}}}		
%----------------------------------------------------------------
\section{Introduction.}

A well-known construction %in representation theory of Lie algebras 
associates a   Kac-Moody algebra $\mathfrak{g}(C)$ to any generalized Cartan matrix $C$ (see e.g. \cite{Kac}). 
 If $C$ is $n\times n$ symmetric matrix it is convenient to encode it  by a graph $\Gamma$ with $n$ vertices, where

$$
\# \text{ of edges between } \al\text{ and }\be = -C_{\al\be},
$$

\noindent  and $C_{\al\be}$ is the corresponding entry of the generalized Cartan matrix. In the classical setting $\mathfrak{g}(C)$ is a simple Lie algebra or an affine Lie algebra and the graph is the corresponding Dynkin diagram or its extended version.

Further, one can define a braid semigroup $B_{\Gamma}$. It is generated by the set of generators  
$\{s_\al~|~\al~\in~\text{Vertices of } \Gamma\}$ subject to \emph{braid relations}. In the simply laced case, that is, when $\Gamma$ has at most one edge between any two vertices, these relations take form: 

$$
\begin{cases}
s_\al s_\be s_\al = s_\be s_\al s_\be, &\text{ if } C_{\al\be}=-1;\\
s_\al s_\be = s_\be s_\al,  &\text{ if } C_{\al\be}=0.
\end{cases}
$$ 

\noindent A further quotient by relations $s_\al^2 = e$ gives a standard realization of the Weyl group $W$. We will often use the term ``braid semigroup'' for the doubled braid semigroup $\BB_\Gamma = B_\Gamma\times B_\Gamma$.

\medskip

In \cite{FG05} a cluster variety ${\mathcal X}_b$ corresponding to an arbitrary element $b\in \BB_\Gamma$ was defined.  Recall that a cluster algebra \cite{FZ01} or a cluster 
variety structure  \cite{FG03}, is described by combinatorial data that consists of a collection of \emph{seeds} related by involutions called \emph{mutations}. In simply laced case a seed is just a quiver with variables attached to its vertices.

Thus, for each   $b\in \BB_\Gamma$, Fock and Goncharov describe a collection of quivers related by mutations. Elements of this collection correspond to reduced presentations  for $b$ via generators.  Hence, for every word $x = s_{i_1}s_{i_2}\ldots s_{i_k}$, there is a quiver $Q(x)$. If $x$ and $y$ are related by a single braid relation, then $Q(y)$ is obtained from $Q(x)$ by an explicit mutation.

In particular any pair of elements $(u,v)$ of a classical 
Weyl group, being lifted to the braid semigroup, provides 
a  cluster Poisson variety structure on the corresponding double Bruhat cell of the related simple Lie group $G^{u,v}\subset G$ \cite{FG05}. The corresponding cluster algebra structure was described in \cite{BFZ03}.

\medskip
Main corollary of results of our paper is a $CY_3$-categorification of this picture for arbitrary simply laced generalized Cartan matrices. This is done using the construction from \cite{KS08} that associates to a \emph{quiver with potential} a $CY_3 \ A_{\infty}$-category equipped with a collection of spherical generators. A potential is a (possibly infinite) linear combination of cyclic paths in the path algebra of the quiver (see section 2 for more precise definitions). Roughly speaking, vertices of the quiver correspond to spherical generators of the category, and cycles in the potential encode higher compositions of $\operatorname{Ext^1}$ in the associated $CY_3\ A_{\infty}$-category. 

Quivers with potential were introduced in \cite{DWZ07}, where authors also described process of mutation extended to potentials. From the point of view of \cite{KS08} these mutations is a combinatorial incarnation of certain transformations of collections of spherical generators. Two quivers with potential are called \emph{mutation-equivalent} if there is a sequence of mutations relating them.% (see e.g. \cite{G16} for more details about this construction).

The   $CY_3$-categorification of cluster varieties is informally depicted in the following diagram:

\begin{figure}[h!]
\begin{minipage}{0.44\textwidth}
\centering
Cluster variety
\vspace{5pt}

$\{$Quivers related by mutations$\}$

\end{minipage}
\begin{minipage}{0.1\textwidth}
\centering
$\longrightarrow$
\end{minipage}
\begin{minipage}{0.44\textwidth}
\centering
$CY_3\ A_{\infty}$-category

\vspace{5pt}

$\left\{\begin{tabular}{c}
Quivers with potentials \\
related by mutations
\end{tabular}\right\}$
\end{minipage}
\end{figure}

\medskip
We want to stress that the space of all   potentials for a given quiver is infinite-dimensional. So a priori a cluster variety gives rise to   an infinite dimensional family of 
${\rm CY}_3$ $A_\infty$-categories. 

 Our main result,  Theorem \ref{main},  tells that  there is an explicitly described \underline{finite dimensional} family of ${\rm CY}_3$ $A_\infty$-categories 
 up to equivalences associated to any element $b\in \BB_\Gamma$.

\begin{theorem}\label{main} Let   $\Gamma$ be a graph with at most one edge between any two vertices. 

Then any element $b$ of the braid semigroup $\BB_\Gamma$ (subject to  technical restrictions (\ref{exclusions},\ref{connect})), gives rise to a  collection  of mutation-equivalent quivers with potential, each corresponding  to a reduced expression of $b$. 
Furthermore, there is  a finite-dimensional family of such collections. 
\end{theorem}

%\noindent 
 %In fact, there is a finite-dimensional space of parameters for collections of quivers with potentials that we construct. 
%As noted above, an immediate corollary of this construction is that there exists a finite dimensional family of $CY_3\ A_\infty$-categories 
Here are a few remarks about the theorem. Quivers with potentials associated to bipartite ribbon graphs were first considered by physicists, see e.g. \cite{FHKVW}. A family of mutation-equivalent quivers with potentials associated to triangulated surfaces was established by Labardini-Fragoso in \cite{L08}. The underlying quivers were considered   in \cite{GSV03}, \cite{FG03}. Any such quiver corresponds to a triangulation of a surface and if two triangulations are related by a flip, then two associated quivers differ by mutation. Quivers with potentials  for the moduli space of framed $n$-dimensional local systems on a decorated surface were introduced in \cite{G16}: from this point of view the $n=2$ case is the one studied in \cite{L08}.

 \medskip
The proof of Theorem \ref{main} is based on careful study of the class of potentials that we call \emph{primitive}, it is preserved by mutations and exhibits particularly nice properties.  
We denote a quiver with potential associated to a presentation $x$ of $b$ by $(Q(\widetilde x), W_x)$. We provide a full description of the space of primitive potentials on $Q(\widetilde x)$ modulo the action of the group of right-equivalences. This space is identified with the second cohomology group of a $2$-dimensional $CW$-complex $\mathcal{C}(\widetilde x)$ whose $1$-skeleton coincides with the quiver (Proposition \ref{mut_pot}). If two quivers are related by mutations, then the associated CW-complexes are homotopy equivalent. Based on these facts the proof of Theorem \ref{main} boils down to the check that the second cohomology class associated to a primitive potential is preserved under mutations.
\medskip

The first application of our construction is when  $\Gamma$ is a simply laced Dynkin diagram. Then we obtain $CY_3$-categories related to Double Bruhat cells. More precisely, quivers $Q(\widetilde x)$ actually describe cluster structure on quotients by the adjoint action of Cartan subgroup:  $G^{u,v}/{\operatorname{Ad} H}$. The potentials we introduce give rise to certain combinatorially described categories. In this case our description of the space of primitive potentials implies that all of them are right-equivalent, it follows that there is a \emph{unique} up to equivalence category associated to a primitive potential for $G^{u,v}/{\operatorname{Ad} H}$.

The geometric origin of this category is a subject for the future research: we believe that it should describe a full subcategory of  a wrapped Fukaya category of  certain non-compact 
Calabi-Yau threefold. 
This program in a different context of quivers with potentials associated to triangulated surfaces \cite{L08} was carried out in detail in \cite{BS13},\cite{S13}. 
It was conjectured in \cite{G16} that the ${\rm CY}_3$ $A_\infty$-categories associated to  
the quivers with potentials  for the moduli space of framed $n$-dimensional local systems on  decorated surfaces 
  are similarly related to  wrapped Fukaya categories of  the open CY threefolds studied in \cite{DDP06}. 

Another application is for $\Gamma = \widehat{A}_n,\ n\geq2$, the extended Dynkin diagram of type A. Cluster structure in this case was studied in detail in \cite{FM14}. In this paper, Fock and Marshakov proved that a class of integrable systems related to dimers on a torus, introduced by Goncharov and Kenyon in \cite{GK11}, can be realized using cluster varieties associated to quivers $Q(\widetilde x)$. Theorem \ref{main} in this case gives a $1$-parametric family of $CY_3$-categories. In fact, the CW-complex that controls the family of potentials is naturally identified with the torus where dimers are defined. The second cohomology group of the torus is indeed $1$-dimensional.
As explained in \cite{GK11}, the cluster integrable systems related to dimers on a torus are parametrized by   Newton polygons $N$. 
It is plausible that our $CY_3$-categories describe full subcategories of a 1-parametric family of   wrapped Fukaya categories of the toric Calabi-Yau threefold associated with the cone over the polygon $N$. 
\medskip
${\rm CY}_3$ $A_\infty$-categories
I am very grateful to Alexander Goncharov for introducing me to this subject and for many useful discussions and remarks.

The structure of the paper is the following: in section 2 we recall precise definitions for quivers with potentials; in section 3 the construction of quivers $Q(\widetilde x)$ is given; section 4 is devoted to the study of primitive potentials on these quivers.

\section{Quivers with potential and their mutations.}
\subsection{Generalities on quivers with potential.}
In this section we provide main definitions related to quivers with potentials and introduce notations.

For us \emph{a quiver} $Q$ is a finite collection of vertices $I$ connected by oriented edges. In our approach  quivers of the interest will be glued from elementary pieces along subsets of special vertices called \emph{frozen}. Arrows between frozen vertices are allowed to have weight $\frac 1 2$, so after gluing of two quivers half-arrows either combine to a usual arrow or cancel out depending on their mutual orientations.
%We do not allow loops - that is arrows going from a vertex to itself. 
More specifically any quiver can be encoded by the following collection:

\begin{definition}
A quiver $Q$ is given by a tuple $(I,I_0,E,E_0, s,t)$ where:
\begin{enumerate}[label=(\roman*)]
	\item $I$ is the set of vertices;
	\item $I_0\subset I$ is the subset of \textbf{frozen vertices};
	\item $E$ is the set of arrows;
	\item $E_0\subset E$ is the subset of \textbf{half-arrows};
	\item these sets are related by maps $s,t:E\rightarrow I$ assigning to every arrow its source and target. 
\end{enumerate}
We require that half-arrows can be adjacent only to frozen vertices, i.e. restrictions of $s$ and $t$ to $E_0$ take values in $I_0$.

\end{definition}

\noindent Sometimes we use notation $Q_0$  (resp. $Q_1$) for the set of vertices (resp. arrows) of quiver $Q$.

The definition of a potential of a quiver relies on the notion of \emph{path algebra} $R\langle Q\rangle$ that we will describe momentarily. Informally, elements of a path algebra $R\langle Q\rangle$ can be thought of as being (linear combinations of) products of composable arrows. By ``composable'' we mean that the source of every successive arrow coincides with the target of the preceding.  Such a product is usually called ``a path''. Our convention is to write products from right to left as for standard composition of maps. Note that source and target maps $s,t$ on $Q$ can be extended to all paths.

Precisely there is the following definition. Fix a base field $\mathbf{k}$ of characteristic zero, and let $R = \mathbf{k}^{Q_0}$ and $A = \mathbf{k}^{Q_1}$ be rings of $\mathbf{k}$-valued functions on vertices and arrows respectively. Note that $R$ is generated by a collection of idempotents $i\in Q_0$ corresponding to $\delta$-functions, and $A$ posesses a natural structure of $R$-bimodule:
$$
i\cdot a \cdot j = \delta_{i,t(e)}\delta_{s(e),j}a\ \ \ \forall i,j\in Q_0, \ a\in Q_1
$$
\noindent To obtain the path algebra we take iterated tensor products of $A$ over $R$.
\begin{definition}
A \textbf{path algebra} of a quiver $Q$ is the graded tensor algebra 
$$R\langle Q\rangle=\bigoplus_{d=0}^\infty \underbrace{A\otimes_R A\ldots\otimes_R A}_{d},$$
\end{definition}
\noindent with the convention that $0$-th graded component is $R$ (its elements are sometimes called ``lazy paths''). Any path algebra has a maximal ideal generated by elements of positive degree, that is by paths of nonzero length, we denote it by $\mathfrak{m}(Q)$. Define the \emph{completed path algebra} $R\langle\langle Q\rangle\rangle$ as the completion of $R\langle Q\rangle$ with respect to powers of the maximal ideal.

One of the advantages of using completed path algebras is that that $R$-homomorphisms of completed algebras have a very simple form. Let $Q$ and $Q'$ be two quivers with the same set of vertices and with arrow spaces $A$ and $A'$. Then structure of these maps is explained by the following proposition which is a little reformulation of Proposition 2.4 from \cite{DWZ07}:
\begin{proposition}\label{right-eq}
Any pair $(\varphi^{(1)}, \varphi^{(2)})$ of $R$-bimodule homomorphisms $\varphi^{(1)}: A\rightarrow A'$ and $\varphi^{(2)}: A \rightarrow \mathfrak{m}(Q')^2$ gives rise to a unique homomorphism of completed path algebras $\varphi:R\langle\langle Q\rangle\rangle\rightarrow R\langle\langle Q'\rangle\rangle$ such that $\varphi |_R=\operatorname{id}$ and $\varphi |_A = (\varphi^{(1)}, \varphi^{(2)})$. Furthermore, $\varphi$ is isomorphism if and only if $\varphi^{(1)}$ is a $R$-bimodule isomorphism.
\end{proposition}
\noindent Thus, automorphisms of path algebras can be understood as certain lower-triangular matrices with respect to the grading.
\vspace{5pt}

Main objects of our study are \emph{potentials} on quivers. They may be understood as (possibly infinite) linear combinations of closed paths considered up to cyclic order of arrows:

\begin{definition}
Let $R\langle\langle Q\rangle\rangle_{\operatorname{cyc}}$ be the linear subspace of the completed path algebra generated by cyclic paths, i.e. products of the form $\prod_{i =1}^n a_{i}$ such that $t(a_1) = s(a_n)$. A \textbf{potential} on $Q$ is an element of $R\langle\langle Q\rangle\rangle_{\operatorname{cyc}}$ considered up to cyclic shift:
$$
a_1 a_2\ldots a_n \longleftrightarrow a_na_1\ldots a_{n-1}
$$
\end{definition}

\noindent\emph{Remark:} It might be helpful to define dual path algebra $R\langle Q^*\rangle$ which is obtained by replacing vector spaces of arrows by their dual vector spaces. Then functionals on this algebra are isomorphic to $R\langle\langle Q\rangle\rangle$. Potentials can be viewed as functionals on $R\langle Q^*\rangle/ [R\langle Q^*\rangle,R\langle Q^*\rangle]$, so it becomes clear why they are considered up to cyclic shifts.

Thus \emph{a quiver with potential} (QP for short) is just a pair $(Q,W)$ where $Q$ is a quiver and $W$ is an element of $R\langle\langle Q\rangle\rangle_{\operatorname{cyc}}$ considered up to cyclic shifts. It immediately follows from definitions that $R$-homomorphisms of completed path algebras induce maps of QPs.
\begin{definition}
Let $(Q,W)$ and $(Q',W')$  be two quivers with potentials on the same set of vertices. We say that they are \textbf{right-equivalent} if there is an $R$-algebra isomorphism $\varphi: R\langle\langle Q\rangle\rangle\rightarrow R\langle\langle Q'\rangle\rangle$ such that $\varphi(W)=W'$.
\end{definition}

Given two quivers with potential on the same set of vertices one can form their direct sum $(Q,W)\oplus (Q',W')$ in an intuitive sense: this is new quiver with potential where all arrow spaces are direct sums of arrow spaces of the summands and the potential is the sum of their potentials. 

\begin{definition}
A quiver with potential is called \textbf{trivial} if it is right-equivalent to the direct sum $\bigoplus (Q_i,W_i)$ where each summand has exactly two arrows $a_i$ and $b_i$ forming a $2$-cycle and the potential is the product of these arrows $W_i =a_ib_i$.
	
If degree $2$ component of the potential of a quiver is zero then it is called \textbf{reduced}.
\end{definition}

Next we quote ``Splitting Theorem'' from \cite{DWZ07} which is essential for mutations:

\begin{theorem}\label{split}
	Every quiver with potential $(Q,W)$ is right-equivalent to direct sum of reduced and trivial quivers with potential $(Q_{\operatorname{red}},W_{\operatorname{red}})\oplus(Q_{\operatorname{triv}},W_{\operatorname{triv}})$. Moreover, right-equivalence class of each of the summands is determined by right-equivalence class of $(Q,W)$.
\end{theorem}

It is often useful to replace a quiver with potential by its reduced part, we will call this procedure \emph{reduction}.

\subsection{Mutation of quivers with potential.}

\begin{figure}
	\begin{minipage}{.33\textwidth}
		\centering
		\captionsetup{width=1.2\textwidth}
		% simple example of mutation: initial quiver
\includegraphics{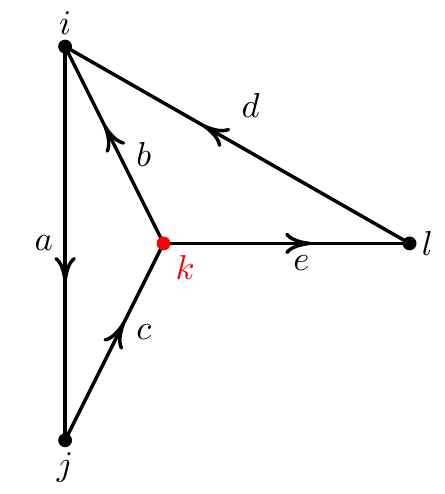}
		\captionof{figure}{$(Q,W=abc)$}
		\label{mutex1}
	\end{minipage}
	\begin{minipage}{.33\textwidth}
		\centering
		\captionsetup{width=1.2\textwidth}
		% premutated quiver
\includegraphics{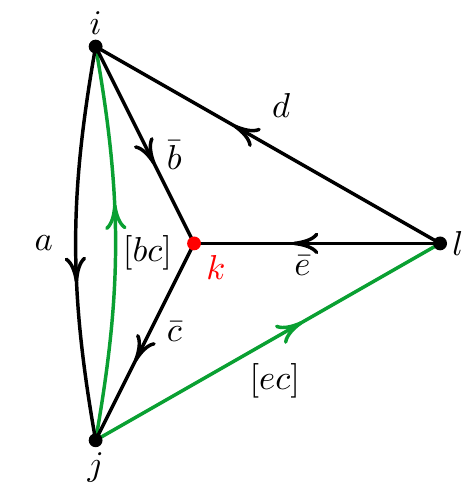}
		\captionof{figure}{$(\widetilde{\mu}_kQ,\widetilde{\mu}_kW)$}
		\label{mutex2}
	\end{minipage}
	\begin{minipage}{.32\textwidth}
		\centering
		\captionsetup{width=1.2\textwidth}
		%reduced premutated quiver = mutated quiver
	
\includegraphics{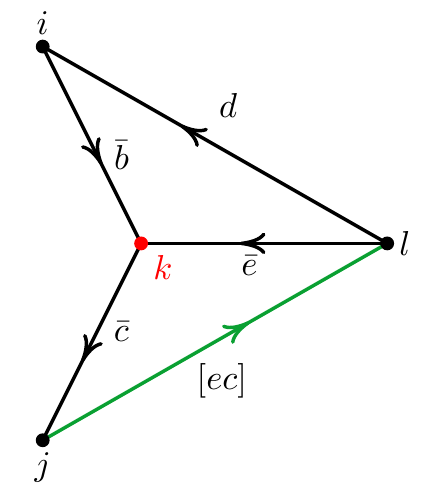}
		\captionof{figure}{$(\mu_kQ,\mu_kW=\bar c \bar e [ec] )$}
		\label{mutex3}
	\end{minipage}
\end{figure}

Throughout this section we assume that quiver $Q$ has no loops and there are no half-arrows. It is known that to every vertex of $Q$ there corresponds a transformation of the quiver called \emph{a mutation}. Remarkably this process can be defined for quivers with potentials too. We briefly recall this story in this section.

Let $k$ be any nonfrozen vertex of $(Q,W)$ that does not belong to a $2$-cycle. Then we are going to produce another quiver with potential denoted by $(\mu_k Q, \mu_k W)$ that is the mutation of $(Q,W)$ at vertex $k$. 
To do so we first construct premutated quiver $(\widetilde{\mu}_k Q,\widetilde{\mu}_k W)$ possibly having $2$-cycles and then we reduce it using Theorem \ref{split}.
%A simple example of this can be found on Figures \ref{mutex1}-\ref{mutex3} (see discussion of it below).

Define arrow spaces of the premutated quiver as follows:

\begin{equation}
\widetilde{\mu}_k A_{i,j}=
\begin{cases}
A_{i,j}\oplus A_{i,k}\otimes A_{k,j}& \text {if $i,j\neq$k;}
\\A_{j,i}^* & \text {otherwise.}
\end{cases}
\end{equation}

\noindent Here $A_{i,j} = i\cdot A\cdot j$ denotes the vector space of arrows from vertex $j$ to $i$. One can interpret this operation on the level of quivers in simple combinatorial terms: in the first case we add new shortcut arrow called $[ab]$ for every $2$-step path $ab$ in $Q$ passing through vertex $k$, in the second case we reverse all arrows adjacent to vertex $k$. Let is indicate this by putting bars on the top, e.g. $\bar a$ stands for reversed arrow $a$.

Now we describe the effect of mutation on the potential. Denote by $[W]$ the potential for the quiver $\widetilde{\mu}_k Q$ obtained by replacing all occurrences of 2-step paths through $k$ by corresponding shortcuts. Then by definition:

\begin{equation}
 \widetilde{\mu}_k W=[W]+\sum_{a,b:\ s(a)=t(b)=k} \bar{b}\bar{a}[ab]
\end{equation}

\noindent The sum over $2$-step paths through $k$ corresponds to triangular terms that are formed by new shortcut arrows $[ab]$ and reversed arrows $\bar a,\ \bar b$. Note that these terms are canonical elements in $ A_{k,s(b)}^*\otimes  A_{t(a),k}^*\otimes A_{t(a),k}\otimes A_{k,s(b)}$. Finally, as mentioned above $\mu_k(Q,W)$ is obtained by removing the trivial part of $(\widetilde{\mu}_k Q,\widetilde{\mu}_k W)$.

\noindent\emph{Example:} In the example in the sequence of Figures \ref{mutex1}-\ref{mutex3} we take potential of the original quiver shown on the left to be $W=abc$. After mutation at vertex $k$, two new shortcut arrows appear in $\widetilde{\mu}_kQ$ (they are shown in green). Furthermore, all arrows attached to $k$ are reversed. We have $$\widetilde{\mu}_kW=[W]+\bar c\bar b[bc]+\bar c\bar e[ec],$$ 
\noindent since $[W]=a[bc]$ and there are two additional triangles corresponding to each shortcut. Finally, to reduce premutated quiver we have to decompose it to the direct sum of its trivial and reduced parts. That can be done as follows: apply right-equivalence $\varphi$ that sends $a\mapsto a -\bar c\bar b$ and leaves other arrows unchanged, then $\varphi(\widetilde{\mu}_kW)=a[bc]+\bar c\bar e[ec]$. Hence the $2$-cycle $a[bc]$ can be split out and it is deleted after reduction. The result of the mutation is shown in Figure \ref{mutex3}, corresponding potential is $\mu_kW = \bar c\bar e[ec]$.
%\bullet\overset{\ a}{\longleftarrow}\underset{k}{\bullet}\overset{\ b}{\longleftarrow}\bullet

\vspace{5pt}
The following two theorems from \cite{DWZ07} justify that mutation is well-defined on right-equivalence classes of QPs and that mutation acts there as an involution.

\begin{theorem}\label{mucorrect}
	The right-equivalence class of $(\mu_k Q, \mu_k W)$ is determined by the right-equivalence class of $(Q,W)$. 
\end{theorem}
\begin{theorem}\label{m^2}
	The correspondence $\mu_k: (Q,W)\longrightarrow (\mu_kQ,\mu_kW)$ acts as an involution on the set of right-equivalence classes of reduced quivers with potentials.
\end{theorem}

\section{Quivers with potentials from words of the free semigroup $\MM_\Gamma$.}
In this section we describe the class of quivers of our interest. Particular examples of such quivers correspond to cluster structures on Poisson leaves of Poisson-Lie groups. Our quivers are naturally obtained by gluing certain elementary pieces, this approach called \emph{amalgamation} was developed by V. Fock and A. Goncharov in \cite{FG05}.
%Their paper explains how these quivers are related to the geometry of Poisson-Lie groups. More specifically, it is shown that all Poisson leaves have cluster structure. These Poisson leaves can be described by elements of amalgamation
 
\subsection{Amalgamation of quivers.}

Here we review amalgamation technique from \cite{FG05}.

Let $(Q^p)_{p \in P} = (I^p,I_0^p, E^p,E_0^p,s^p, t^p)_{p\in P}$ be a collection of quivers labelled by a finite set $P$. In order to amalgamate them we need to fix \emph{gluing data} $\mathcal{P}$ that consists of a set $I^{\mathcal{P}}$ and a collection of maps $\iota^p:I^p\hookrightarrow I^{\mathcal{P}}$ that cover $I^{\mathcal{P}}$. If some $k\in I^{\mathcal{P}}$ is covered by more than one $\iota^r$ then all its preimages are required to be frozen.

Assume that all quivers have no $2$-cycles, then arrows in $Q^p$ can be uniquely described by skew-symmetric adjacency matrix whose entries are half-integers $\varepsilon^r_{i,j},\ i,j \in I^p$ equal to the number of arrows from $j$ to $i$ minus the number of arrows from $i$ to $j$  in $Q^p$  (half-arrows are taken with weight $\frac 12$).

\begin{definition} \textbf{The amalgamation} of the collection $(Q^p)_{p\in P}$ corresponding to gluing data  $\mathcal{P}$ is the new quiver $Q^\mathcal{P}$ with the set of vertices $I^{\mathcal{P}}$ and the adjacency matrix given by 
$$
\varepsilon^\mathcal{P}_{i,j} = \sum\varepsilon^p_{i',j'}
$$
\noindent where the sum is taken over the set of pairs $\{(i',j')\mid i',\ j'\in I^p,\ \iota^p(i') = i,\  \iota^p(j') = j\}$. 
\end{definition}

\noindent In other words, we take all arrows whose endpoints are mapped to vertices $i$ and $j$ in $I^{\mathcal{P}}$and compute the resulting number of arrows taken with appropriate signs and weights.
The set of frozen vertices $I^\mathcal{P}_0\subset I^\mathcal{P}$ is defined as the union of images of frozen vertices. After amalgamation it may happen that there are frozen vertices with no adjacent half-arrows. If this is the case we can defrost them, i.e. exclude them from $I^\mathcal{P}_0$. Since it is usually evident from the context what arrows can be defrosted we usually omit these details.

%In fact, for all amalgamated quivers in this paper there will be at most one arrow connecting any two vertices. Moreover, during amalgamation all interesting entries of the resulting adjacency matrix will correspond either to gluing or to cancellation of two half-arrows. Thus the sums for $\varepsilon^\mathcal{P}_{i,j}$ will have form $\pm\frac 12 \pm \frac 12$.

\subsection{Construction of quivers $Q(x)$.}

\begin{wrapfigure}{l}{0.4\textwidth}
  \vspace{-20pt}
  \centering
	\includegraphics{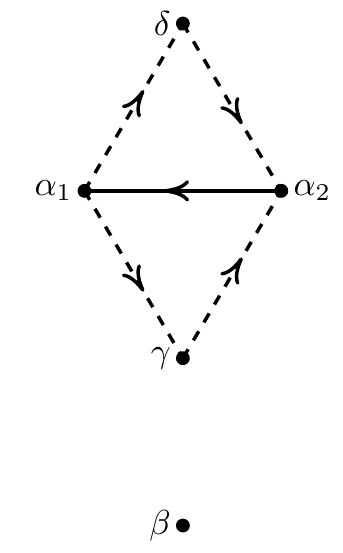}	
	 \vspace{-10pt}
	\captionof{figure}{Elementary~quiver}
	\label{elementary}
  \vspace{-30pt}
\end{wrapfigure}
%[may have no restrictions on \Gamma at all, e.g \hat A_1 works (but mutations don't...), but then need to generalize defns. several half-arrows joining two vertices may be required]
In this section a class of quivers generalizing \cite{FG05} is defined. Let $\Gamma$ be a graph with no loops and at most one edge connecting any two vertices, denote its set of vertices by $\Pi$. In geometric applications related to simply laced Poisson-Lie groups and loop groups this graph is just the corresponding Dynkin diagram or its extended version and elements of $\Pi$ are identified with simple roots. We define a free semigroup $\MM_\Gamma$ and construct a collection of quivers with potential $(Q(x),W(x))$ associated to every element $x\in \MM_\Gamma$. These quivers admit a number of mutations that preserve the collection. Such mutations are related to braid relations exchanging generators of the free semigroup.

\vspace{5pt}
The semigroup $\MM_\Gamma$ is generated by letters $\{s_\alpha, s_{\bar{\al}}\mid \alpha\in\Pi\}$. Denote by $C$ the associated generalized Cartan matrix. Its entries are defined as follows:

\begin{equation}
\begin{cases}
C_{\al\be} = -1 \text{, if } \al \text{ and } \be \text{ are connected by an edge in } \Gamma;\\
C_{\al\be} = 2 \text{, if } \al=\be;\\
C_{\al\be} = 0 \text{, otherwise.}
\end{cases}
\end{equation}

\noindent Braid semigroup $\mathfrak{B}_\Gamma$ is the quotient of $\MM_\gamma$ by relations of three types:

\begin{figure}[h!]
\begin{minipage}{0.8\textwidth}
\begin{enumerate}[label=(\roman*)]
	\item \label{braid_rel1}$s_\alpha s_\beta=s_\beta s_\alpha,\  s_{\bar{\al}}s_{\bar{\be}}=s_{\bar{\be}}s_{\bar{\al}}\ \text{if} \ \ C_{\alpha\beta}=0$
	\item \label{braid_rel2}$s_\alpha s_\beta s_\alpha=s_\beta s_\alpha s_\beta,\ s_{\bar\alpha} s_{\bar\beta} s_{\bar\alpha}=s_{\bar\beta} s_{\bar\alpha} s_{\bar\beta}\ \text{if}\ \ C_{\alpha\beta}=-1$
	\item \label{braid_rel3}$s_\al s_{\bar{\be}}=s_{\bar{\be}}s_\al$
\end{enumerate}
\end{minipage}
\begin{minipage}{0.19\textwidth}
\begin{equation}\label{braid_rels}
\end{equation}
\end{minipage}
\end{figure}

In \cite{FG05} a quiver $Q(x)$ was associated to any word $x\in \MM_D$ (where $D$ is arbitrary Dynkin diagram). Moreover, it is shown there that if two words differ by relation of the first type then corresponding quivers are either identical or differ by explicit mutation. %If two words differ by relation of the second type, then corresponding quivers are related by explicit mutation. For the relation of the third type quivers are either identical or related by mutation depending on whether $\al=\be$ or not. In this section we introduce potentials for quivers of this class that are related in a similar way.
It follows from this construction that one can associate cluster variety to any element of the braid semigroup $\BB_D$. Our main goal is to extend this construction to quivers with potentials.
\medskip

Take an element $x = s_{i_1}s_{i_2}\ldots s_{i_k} \in \MM_\Gamma$  (subindices $i_p$ can be $\al$ or $\bar{\al}$).
To construct a general quiver $Q(x)$ we first associate a quiver $Q(s_\al)$ 
 to every generator $s_\al\in \MM_\Gamma$, such quivers will be called \emph{elementary}. After that we describe $Q(x)$ as amalgamation of elementary quivers $(Q(s_{i_p}))_{p=1\ldots k}$.
 
For any $\al\in\Pi$ the set of vertices of elementary quiver $Q(s_\al)$  is identified with $I_\al  = (\Pi - \al)\cup \al'\cup \al''$. There are no arrows between $i,j\in I_\al$ unless one of the vertices is $\al'$ or $\al''$, in the latter case the corresponding entry in the adjacency matrix $\varepsilon(\al)$ is defined by formulas below (see Fig. \ref{elementary}).  Elementary quiver $Q_{s_{\bar\al}}$ is obtained by reversing all arrows in~$Q_{s_\al}$.

\begin{equation}
\varepsilon(\al)_{\al'\al''} = -1,\ \ \varepsilon(\al)_{\al'\be} = \frac{C_{\al\be}}2,\ \ \varepsilon(\al)_{\al''\be} = -\frac{C_{\al\be}}2,\ \be\in (\Pi - \al).
\end{equation}

\medskip
In order to specify the gluing data for the collection $(Q(s_{i_j}))_{j =1\ldots k}$ we define the set of vertices of amalgamated quiver and introduce suitable notation on the way. Let $n_\al$ be the number of letters $s_\al$ and $s_{\bar{\al}}$ in the expression for $x$. Then the set of vertices of the amalgamated quiver $Q(x)$ is by definition identified with the set of pairs

\begin{equation}
I_x=\bigg\{{\al \choose r}\Big|\al\in \Pi, 0\leq r \leq n_\al\bigg\},
\end{equation}

\noindent where ${\al \choose i}$ is just our a notation for the corresponding pair. It is convenient to use the following terminology:

\begin{definition} \textbf{Decoration} on the set of vertices of $Q(x)$ or $Q(\widetilde x)$ is the map from $I_x$ to $\Pi$ that sends vertex ${\al \choose r}$ to $\al\in \Pi$. We say that a vertex of the quiver is decorated by $\al$ if its image under the decoration map is $\al$. The set of all vertices decorated by one $\al\in\Pi$ form a \textbf{perch}, and the number of vertices on $\al$-perch in $Q(\widetilde x)$ is $n_\al$.
\end{definition}

To describe embeddings of vertex sets of elementary quivers  $Q(s_{i_j})$ to $I_x$ in the gluing data, we introduce additional indexing on letters of $x$. Write $s^{(r)}_{\al}$ (resp. $s^{(r)}_{\bar\al}$) for $r$-th occurrence of $s_{\al}$ or $s_{\bar{\al}}$ in $x$. For instance, the word $s_\al s_\al s_{\bar{\al}} s_{\bar{\be}} s_\al s_\be$ acquires indices $s_\al^{(1)} s_\al^{(2)}s_{\bar{\al}}^{(3)}s_{\bar{\be}}^{(1)} s_\al^{(4)} s_\gamma^{(1)}s_\be^{(2)}$, here letters $s_\al,s_{\bar{\al}}$ are numbered from $1$ to $4$; letters $s_\be,s_{\bar{\be}}$ are numbered from $1$ to $2$; and there is a single letter $s_\gamma$.

Now consider some quiver $Q(s_\al^{(r)})$, as mentioned above its vertex set is $(\Pi-\{\al\})\cup\{\al',\al''\}$. Corresponding embedding $I_{s_\al^{(r)}} \longrightarrow I_x$ maps $\al'$ and $\al''$ to ${\alpha \choose r-1}, {\alpha \choose r}$. Any other vertex $\be\neq \al',\al''$ in $Q(s_\al^{(r)})$ is mapped to ${\be \choose j}$, where $j$ is the number of letters ${s_\be}$ or $s_{\bar\be}$ to the left of $s_\al^{(r)}$, or $j=0$ if there are no such letters. The procedure for elementary quivers for opposite letters $s_{\bar\al}$ goes along same lines.

\begin{figure}
	\centering
	\begin{minipage}{.5\textwidth}
		\centering
		\captionsetup{width=1.2\textwidth}
		\includegraphics{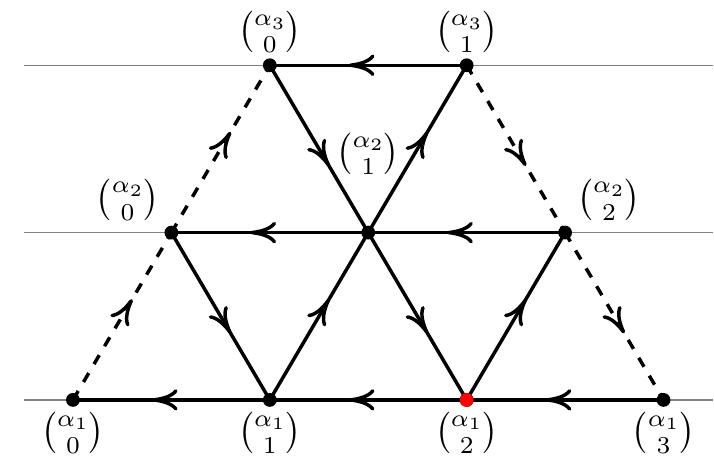}
		\captionof{figure}{$Q_x$ for $x = s_{\al_1}s_{\al_2}s_{\al_3}s_{\al_1}s_{\al_2}s_{\al_1}$}
		\label{A3long}
	\end{minipage}%
	\begin{minipage}{.5\textwidth}
		\centering
		\captionsetup{width=1.2\textwidth}
		\includegraphics{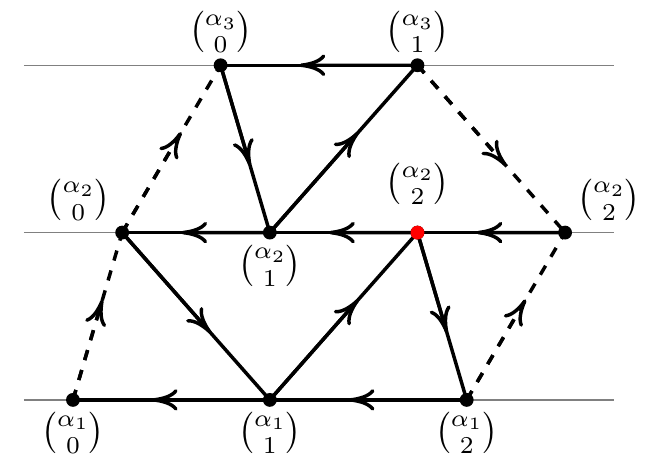}
		\captionof{figure}{$Q_y$ for $y = s_{\al_1}s_{\al_2}s_{\al_3}s_{\al_2}s_{\al_1}s_{\al_2}$}
		\label{mA3long}
	\end{minipage}
\end{figure}

\medskip
\noindent \emph{Example:} In Figures \ref{A3long},\ref{mA3long} graph $\Gamma$ is Dynkin diagram $A_3$, and $x, y\in \MM_{A_3}$ are two lifts of the longest element of the Weyl group: $x  = s_{\al_1}s_{\al_2}s_{\al_3}s_{\al_1}s_{\al_2}s_{\al_1}$, $y = s_{\al_1}s_{\al_2}s_{\al_3}s_{\al_2}s_{\al_1}s_{\al_2}$. Informally, to obtain $Q(x)$ or $Q(y)$ one can imagine that all elementary quivers $Q_{s_i}$ are placed on parallel perches labelled by elements  of $\Gamma$ in the order of their occurrence in the word (see Fig. \ref{A3long_split},\ref{mA3long_split}). We say that perches labelled by $\al,\be\in \Gamma$ are neighbors if $C_{\al\be}=-1$. After elementary quivers are placed on perches, we see that there are horizontal arrows, one for each letter in $x$, and various half-arrows joining neighboring perches. Horizontal arrows lie on perches and go from right to left for generators $s_\al$ and in the opposite direction for generators $s_{\bar{\al}}$. To obtain $Q_x$ compress the picture horizontally, then vertices lying on the same perch and not separated by a horizontal arrow collapse together,  and all half-arrows glue together or annihilate each other in a natural way. 
Note how labelled vertices in previous two figures collapsed pairwise in resulting quivers, also one sees that the interior of amalgamated quiver has only full arrows.

%assembling long word from elementary pieces
\begin{figure}
	\centering
	\includegraphics{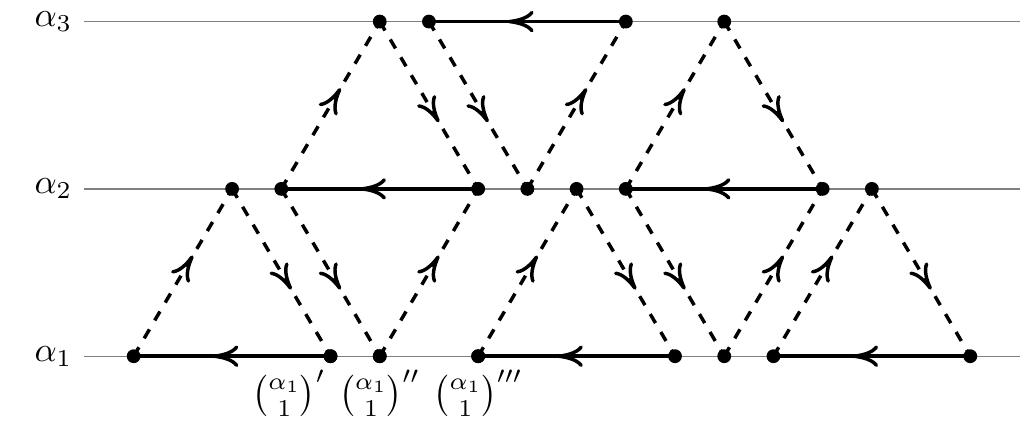}
	\captionof{figure}{$x = s_{\al_1}s_{\al_2}s_{\al_3}s_{\al_1}s_{\al_2}s_{\al_1}$}
	\label{A3long_split}
\end{figure}
\begin{figure}
	\centering
	\includegraphics{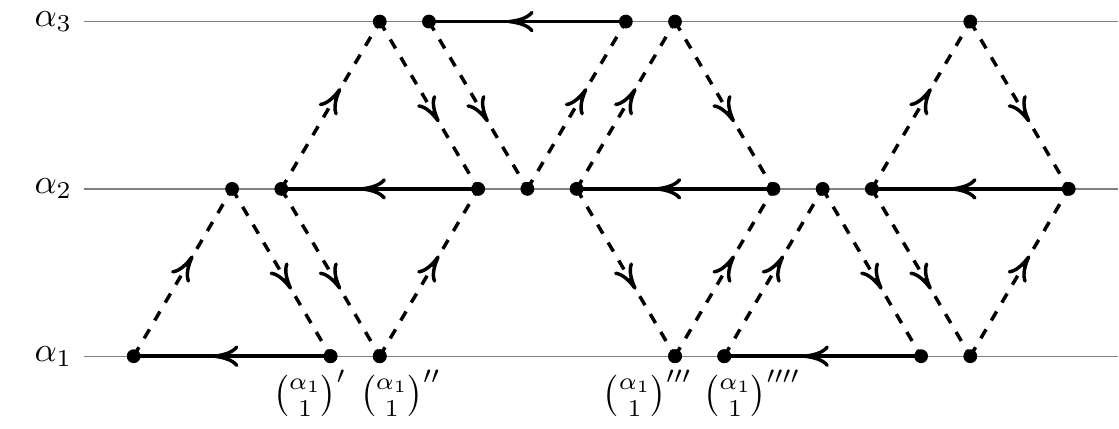}
	\captionof{figure}{$y = s_{\al_1}s_{\al_2}s_{\al_3}s_{\al_2}s_{\al_1}s_{\al_2}$}
	\label{mA3long_split}
\end{figure}

In fact, this description works more generally for any product $xy$ in $\MM_D$. Assuming that we already have $Q_x$ and $Q_y$, we obtain $Q_{xy}$ by placing two parts on perches as before and then for every perch identifying rightmost vertex of $Q_x$ with leftmost vertex of $Q_y$. There is an example of that in Figures \ref{amalg_split},\ref{amalg}.

% two general quivers amalgamation
\begin{figure}
	\centering
	\begin{minipage}{.5\textwidth}
		\centering
		\captionsetup{width=1\textwidth}
\includegraphics{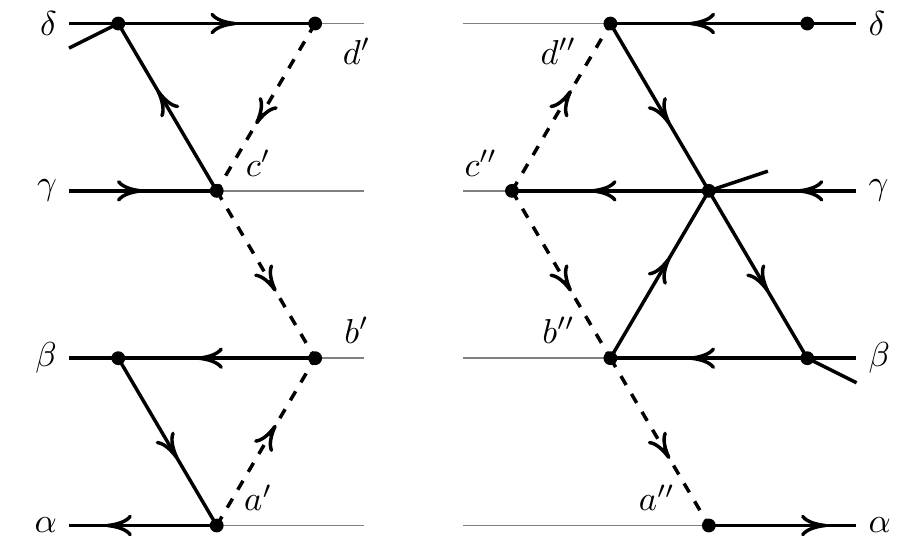}
		\captionof{figure}{$Q_x$ and $Q_y$ placed on perches}
		\label{amalg_split}
	\end{minipage}%
	\begin{minipage}{.5\textwidth}
		\centering
	\includegraphics{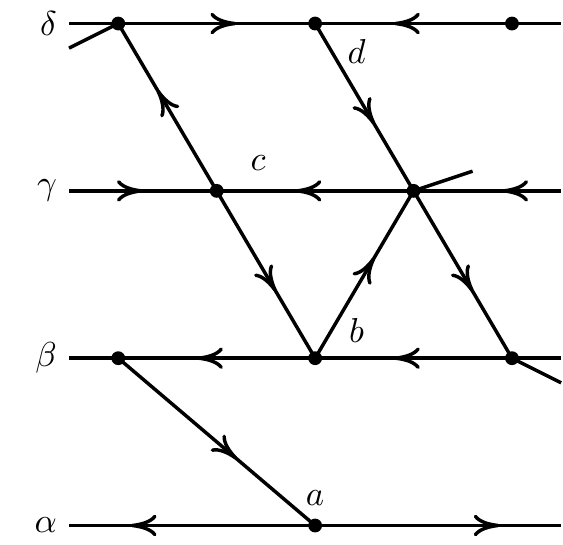}
		\captionof{figure}{$Q_{xy}$}
		\label{amalg}
	\end{minipage}
\end{figure}

\medskip
When $\Gamma$ is a simply laced Dynkin diagram these quivers describe cluster structure on double Bruhat cells of the corresponding group. In order to pass to quotient by adjoint action of Cartan torus $H$ one has to identify vertices ${\al \choose 0}$ and $\al \choose \ n_\al$ for all $\al\in \Gamma$ thus reducing number of vertices by $\dim H$. This motivates our next definition of $Q(\widetilde{x})$. The vertex set is identified with:

\begin{equation}
I_{\widetilde x}=\bigg\{{\al \choose r}\Big|\al\in \Gamma, 0\leq r \leq n_\al-1 \bigg\},
\end{equation}

\noindent and embeddings for the gluing data are defined as before except for the quivers $Q(s^{n_\al}_\al)$, for which vertices $\al'$ and $\al''$ are mapped to ${\alpha \choose n_\al-1}, {\alpha \choose 0}$.

It is immediate from these definitions that if $x$ and $x'$ are cyclic permutations of each other, then quivers $Q(\widetilde x)$ and $Q(\widetilde{x'})$ are identical. Indeed, it is convenient to think of $Q(\widetilde x)$ as being a cyclic closure of the quiver $Q(x)$. Note that $Q(x)$ always has some half-arrows on the boundary, but its cyclic cousin has only ordinary ``full'' arrows.

\medskip
Recall relations (i)-(iii) for the semigroup $\BB_\Gamma$. Next proposition is taken from \cite{FG05} and relates quivers $Q(x)$ (or $Q(\widetilde x)$) corresponding to different lifts an element $b\in\BB_\Gamma$ to $\MM_{\Gamma}$ (see e.g Fig. \ref{A3long},\ref{mA3long}). Its proof and more detailed statement is revealed in subsequent lemmas.

\begin{proposition}\label{braid_mutations}
Let $x$ and $y$ be two lifts of $b\in\BB_\Gamma$ to $\MM_\Gamma$. Then $Q(x)$ (resp. $Q(\widetilde x)$) is mutation equivalent to $Q(y)$ (resp. $Q(\widetilde y)$).
\end{proposition}

\noindent It follows from lemmas \ref{lemma_ab}, \ref{lemma_aba}, \ref{lemma_aa} below.

\begin{lemma}
\label{lemma_ab}
If words $x$ and $y$ differ by a single relation of the form \ref{braid_rel1} or by a single relation of the form \ref{braid_rel3} with $\al\neq\be$, then associated quivers $Q(x)$ and $Q(y)$ are identical. 
\begin{flushright}
$\square$
\end{flushright}
\end{lemma}

\begin{lemma}\label{lemma_aba}
If words $x$ and $y$ differ by a single relation of the form \ref{braid_rel2}, corresponding to exchange of fragments $s_{\al}^{(r)}s_\be^{(q)}s_{\al}^{(r+1)}\longleftrightarrow s_{\be}^{(q)}s_\al^{(r)}s_{\be}^{(q+1)}$, then $Q(y)$ is the mutation of $Q(x)$ at the vertex ${\al \choose r}$.
\end{lemma}

\begin{proof}
First we describe all arrows adjacent to ${\al \choose r}$. From the fragment 
$s_{\al}^{(r)}s_\be^{(q)}s_{\al}^{(r)}$ of $x$ it is clear that there are four arrows: ${\al \choose r}\longrightarrow {\al \choose r-1}$, ${\al \choose r+1}\longrightarrow {\al \choose r}$, ${\be \choose q-1}\longrightarrow {\al \choose r}$ and ${\al \choose r}\longrightarrow {\be \choose q}$. We claim that there are no other arrows adjacent to ${\al \choose r}$. Indeed, for any other $\be'$ there can be an arrow joining ${\be'\choose q'}$ and ${\al \choose r}$ only if $C_{\al\be'} = -1$. But during amalgamation these arrows come from $Q(s_{\al}^{(r)})$ and $Q(s_{\al}^{(r+1)})$, each of them has one half-arrow and they cancel each other in $Q(x)$. 
Thus, the mutation at ${\al \choose r}$ before cancelling oriented 2-cycles introduces only four new arrows: 
${\be\choose q-1}\longrightarrow{\be\choose q}$ (always cancells out in a 2-cycle), ${\al\choose r+1}\longrightarrow{\al\choose r-1}$ (always persists) and ${\be\choose q-1}\longrightarrow{\al\choose r-1}$, ${\al\choose r+1}\longrightarrow{\be\choose q}$. It is straightforward to check that the latter two arrows persist or cancel after mutation in a way consistent with the exchange relation \ref{braid_rel2}. Consequently, the quiver $\mu_{\al \choose r}Q(x)$ coincides with $Q(y)$ after renaming vertices in the following way:

$${\al\choose r} \mapsto {\be \choose q},\ {\al\choose r'} \mapsto  {\al\choose r'-1}, \ {\be\choose q'} \mapsto  {\be\choose q'+1}\ \forall r'>r,\ q'\geq q$$
\end{proof}

\begin{lemma}\label{lemma_aa}
If words $x$ and $y$ differ by a single relation of the form \ref{braid_rel3}, corresponding to exchange of fragments $s_{\al}^{(r)}s_{\bar\al}^{(r+1)}\longleftrightarrow s_{\bar\al}^{(r)}s_{\al}^{(r+1)}$, then $Q(y)$ is the mutation of $Q(x)$ at the vertex ${\al \choose r}$.
\end{lemma}
\begin{proof}
Let $\be_1,\be_2, \ldots\be_k$ be the set of all elements in $\Gamma$ such that $C_{\al\be_j} =  -1$. Then since $x$ contains two subsequent letters $s_{\al}^{(r)} s_{\bar\al}^{(r+1)}$, the amalgamated quiver $Q(x)$ has exactly $k+2$ arrows adjacent to ${\al \choose r}$: ${\al\choose r}\longrightarrow{\al\choose r-1}$, ${\al\choose r}\longrightarrow{\al\choose r+1}$, ${\be_j\choose q_j}\longrightarrow{\al\choose r}$, where $1\leq j\leq k$ and $q_j$ is the number of letters $s_{\be_j}$ and $s_{\bar{\be_j}}$ to the left of $s_{\al}^{(r)}$.

Thus the mutation at vertex ${\al \choose r}$ before cancellation of oriented $2$-cycles introduces $2k$ new arrows:  ${\be_j\choose q_j} \longrightarrow {\al\choose r-1}$, ${\be_j\choose q_j} \longrightarrow {\al\choose r+1},\ \forall 1\leq j \leq k$. Direct check shows that these arrows appear or cancel in 2-cycles after mutation in the way consistent with exchanging letters $s_{\al}^{(r)}$ and $s_{\bar\al}^{(r+1)}$. This computation shows that $Q(y)$ can be identified with $\mu_{\al\choose r}Q(x)$.
\end{proof}

Note that proposition \ref{braid_mutations} shows that one can associate cluster variety to an element of~$\BB_\Gamma$.

\section{Primitive potentials on quivers $Q(\widetilde x)$.}

%[additional assumptions on the form of $x$: at least one letter of every type, every type of cycle exists]

%To define the potential on the quiver $Q(\widetilde x)$ we use the following terminology.

We demand that quivers in this section satisfy the condition (\ref{exclusions}) that ensures absence of loops and $2$-cycles in $Q(\widetilde x)$. Condition (\ref{connect}) is equivalent to connectedness of the quiver.

\begin{equation}\label{exclusions}
\text{$\forall\al\in\Gamma$ the number of letters $s_\al$ and $s_{\bar\al}$ in $x$ is at least $3$;}
\end{equation}
\begin{equation}\label{connect}
\text{$\forall \al,\be\in \Gamma$ with $C_{\al\be} = -1$ there exists at least one $\al\bar\be$-cycle in $Q(\widetilde x)$}
\end{equation}

Note that condition (\ref{exclusions}) guarantees that all spaces of arrows in the quiver are one-dimensional. So after choosing a basis in each space of arrows we can identify elements of path algebras with products of arrows with coefficients in the field $\kk$. 

%albe cycle
\begin{figure}
	\centering
	\captionsetup{width=0.8\textwidth}
	\includegraphics{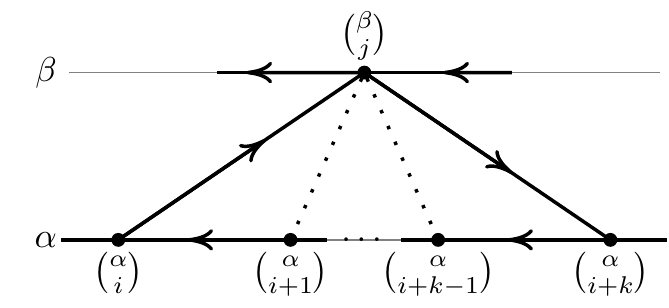}
	\captionof{figure}{$\al\bar{\be}$-cycle for $s_\al^{(i+1)}\ldots s_\al^{(i+k)}$}
	\label{cycle1}
\end{figure}
\begin{figure}
	\centering
	\captionsetup{width=0.8\textwidth}
	\includegraphics{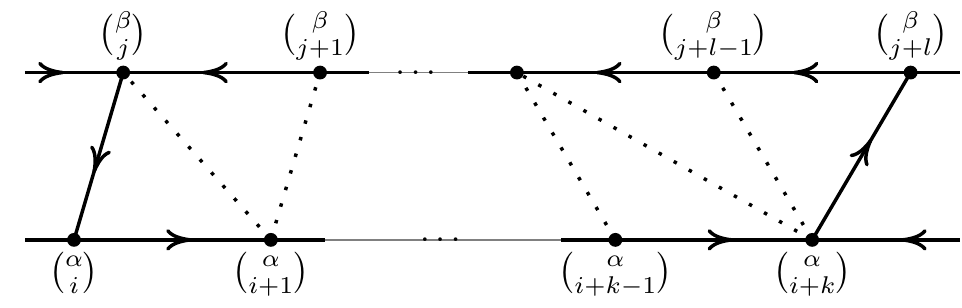}
	\captionof{figure}{$\be\bar{\al}$-cycle for  $s_{\bar{\al}}^{(i+1)}\ldots s_{\bar{\al}}^{(i+k)}s_\be^{(j+1)}\ldots s_\be^{(j+l)}$}
		\label{cycle2}
\end{figure}

Cycles appearing in potentials of our interest are defined as follows:

\begin{definition} Fix two elements $\al,\be\in \Gamma$ with $C_{\al\be} = -1$. Choose a maximal subsequence of $x$ consisting of letters $s_\al$ and $s_{\bar\be}$ such that there are no letters $s_{\bar\al}$ and $s_\be$ in between.  This subsequence is composed by letters $s_\al^{(i+1)}\ldots s_\al^{(i+k)}, s_{\bar\be}^{(j+1)},\ldots s_{\bar\be}^{(j+l)}$ for some $i,j,k,l$ (where all superscripts for $s_\al$ and $s_{\bar\al}$ are considered modulo $n_\al\ZZ$). Moreover, by its maximality the cycle 

\begin{equation}\label{ab-cycle}
\left[{\be \choose j}\rightarrow{\al \choose i}\rightarrow{\al \choose i+1}\rightarrow\ldots\rightarrow{\al \choose i+k}\rightarrow{\be \choose j+l}\rightarrow{\be \choose j+l-1}\rightarrow\ldots\rightarrow{\be \choose j}\right]
\end{equation}

\noindent is well defined (in other words, all arrow spaces involved are one-dimensional). We call a cycle constructed in this way  from any maximal subsequence an \textbf{$\al\bar\be$-cycle} (see Fig. \ref{cycle1}, \ref{cycle2}).
\end{definition}

\begin{definition}
A potential on $Q(\widetilde x)$ is called \textbf{primitive} if it is a sum of \emph{all} possible $\al\bar{\be}$-cycles in $Q(\widetilde x)$ ($\forall\al,\be\in\Gamma$, s.t. $C_{\al\be}=-1$), and all such cycles are taken with nonzero coefficients.
\end{definition}

\medskip
We will show that the space of primitive potentials on $Q(\widetilde x)$ modulo right-equivalences has a very simple structure. The first step for understanding invariants of primitive potentials under the action of the group of right-equivalence is Proposition \ref{phi1} below. Recall by Proposition \ref{right-eq} any right-equivalence $\varphi$ is determined by the two components $(\varphi^{(1)}, \varphi^{(2)})$ of its action on arrow spaces. 

\begin{proposition}\label{phi1}
Let $W$ be a primitive potential on $Q(\widetilde x)$ satisfying conditions (\ref{exclusions},\ref{connect}). Assume that the right-equivalence $\varphi$ maps $W$ to another primitive potential $W'$. Then $W' =\varphi(W) =  \varphi^{(1)}(W)$.

In other words, the action of the group of right-equivalences on the space of primitive potentials depends only on its ``linear'' part.
\end{proposition}

Here are several remarks about the proposition. Since by definition $\varphi^{(2)}$ maps any arrow to a path of length $\geq 2$ with same initial and terminal points, the action of $\varphi$ on chordless cycles obviously does not depend on $\varphi^{(2)}$. This argument would be sufficient to prove the proposition if all $\al\bar\be$-cycles were chordless. However, this is not true in general. For instance, a non-chordless cycle appears in the quiver $Q(\widetilde x)$ for $\Gamma = A_2$ and $x = s_1s_{\bar 2}s_1s_2s_{\bar1}s_1$. Extra identification of vertices in $Q(x)$ does create some chords in $\al\bar\be$-cycles. Using the same example one can also verify that if $\varphi$ preserves a primitive potential, it is not necessarily true that $\varphi^{(2)} = 0$.

Proof of Proposition \ref{phi1} is based on two steps: we use Lemma \ref{simple_chords}  to describe all possible configurations of chords appearing in $Q(\widetilde x)$, then using Lemma \ref{phiW} we relate action of right-equivalences on a potential to combinatorics  of chords.

The following definition of chords is used to include possibility of repeating vertices in a cycle.

\begin{definition}\label{chord}
A \textbf{chord} in a cycle $a_1a_2\ldots a_k$ of quiver $Q$ is a triple $(b,i,j)$, where $b$ is an arrow in $Q$ such that $t(b) = t(a_i)$ and $s(b) = s(a_j)$ with $i\neq j$. 
\end{definition}

\begin{lemma}\label{simple_chords}
Let $(b, i, j)$ be a chord in $\al\bar\be$-cycle $C$ in the quiver $Q(\widetilde x)$ satisfying (\ref{exclusions},\ref{connect}). Then indices $(i,j)$ are uniquely determined by the arrow $b$.
\end{lemma}
\begin{proof}
By applying cyclic permutation to letters of $x$ we can assume that maximal subsequence of letters $s_\al$ and $s_{\bar\be}$ associated to $C$ is formed by letters $s_{\al}^{(1)},\ldots s_\al^{(k)}$ and $s_{\bar\be}^{(1)},\ldots s_{\bar\be}^{(l)}$, where without loss of generality $k\geq 1, l< n_\be$. Recall that quiver $Q(\widetilde x)$ can be obtained from $Q(x)$ by identifying vertices $\al \choose n_\al$ with ${\al \choose 0}$ for all $\al\in \Gamma$. Since any $\al\bar\be$-cycle was chordless in $Q(x)$ there are only four possible chords for $C$ up to direction: horizontal ${\al \choose n_\al -1}\leftrightarrow {\al \choose 0},\ {\be\choose 0}\leftrightarrow{\be\choose n_\be -1}$, and diagonal ${\al\choose 0}\leftrightarrow{\be\choose l}$ or ${\al\choose k}\leftrightarrow{\be\choose 0}$. Thus, all possible combinations of chords in $C$ in $Q(\widetilde x)$ are exhausted by the following list (some items may lead to similar combinations of chords):

{\centering
\begin{table}[ht]
\begin{tabular}{ccc}
	\begin{subfigure}{0.33\textwidth}\centering		
		%case (i)
		\includegraphics{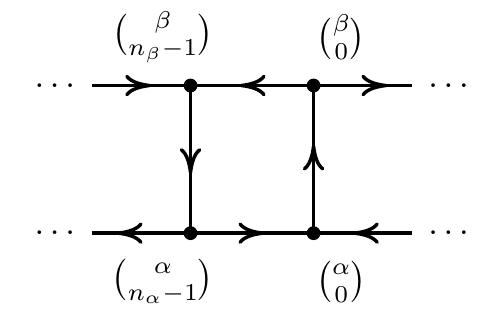}
		\caption*{Case (i)}\label{case_i}
	\end{subfigure}&	
	\begin{subfigure}{0.33\textwidth}\centering
		%case (ii)
		\includegraphics{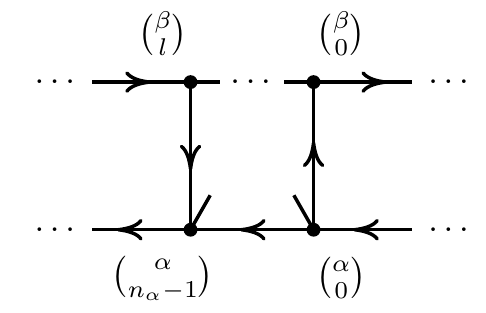}
		\caption*{Case (ii)}\label{case_ii}
	\end{subfigure}&	
	\begin{subfigure}{0.33\textwidth}\centering		
		%case (iii)
		\includegraphics{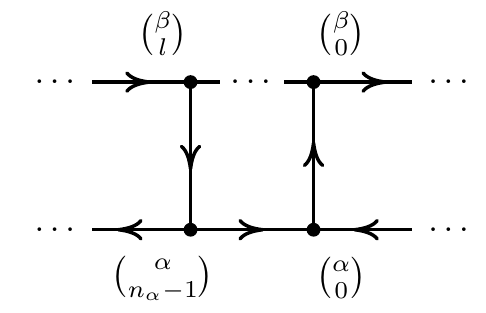}
		\caption*{Case (iii)}\label{case_iii}
	\end{subfigure}\\	
\newline
	\begin{subfigure}{0.33\textwidth}\centering		
		%case (iv)
		\includegraphics{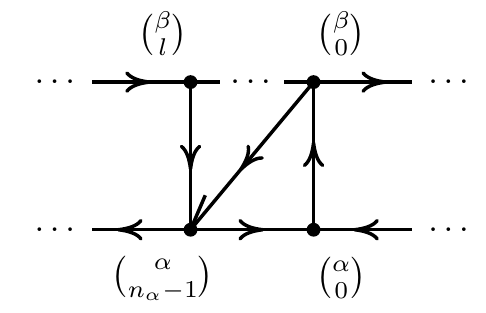}
		\caption*{Case (iv)}\label{case_iv}
	\end{subfigure}&
	
	\begin{subfigure}{0.33\textwidth}\centering		
		%case (v)
		\includegraphics{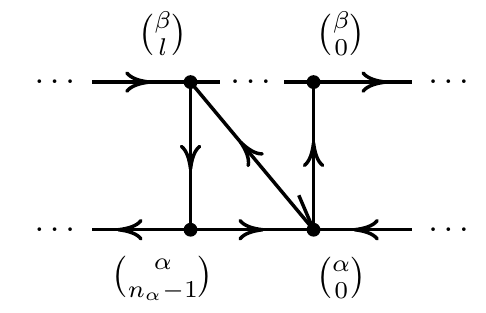}
		\caption*{Case (v)}\label{case_v}
	\end{subfigure}&
	
	\begin{subfigure}{0.33\textwidth}\centering		
		%case (vi)
		\includegraphics{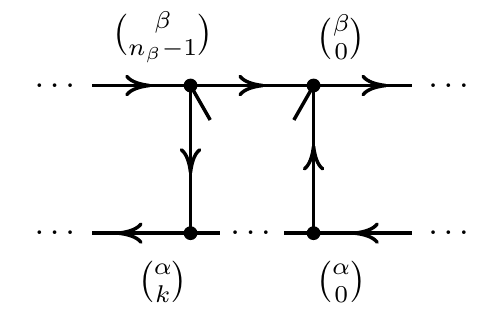}
		\caption*{Case (vi)}\label{case_vi}
	\end{subfigure}\\
\newline
	\begin{subfigure}{0.33\textwidth}\centering		
		%case (vii)
		\includegraphics{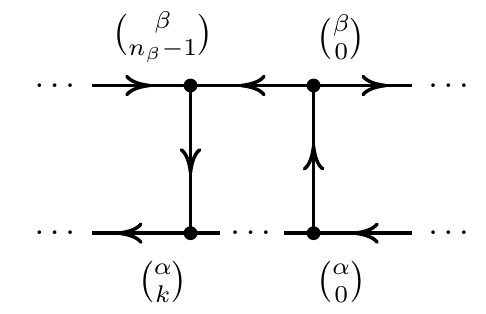}
		\caption*{Case (vii)}\label{case_vii}
	\end{subfigure}&
	
	\begin{subfigure}{0.33\textwidth}\centering		
		%case (viii)
		\includegraphics{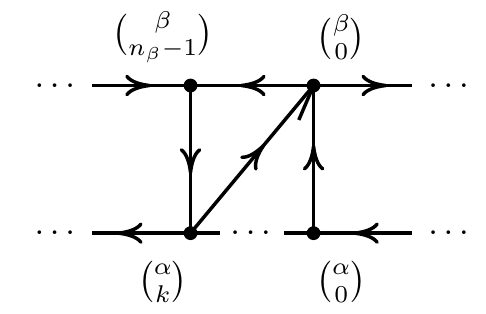}
		\caption*{Case (viii)}\label{case_viii}
	\end{subfigure}&
	
	\begin{subfigure}{0.33\textwidth}\centering		
		%case (ix)
		\includegraphics{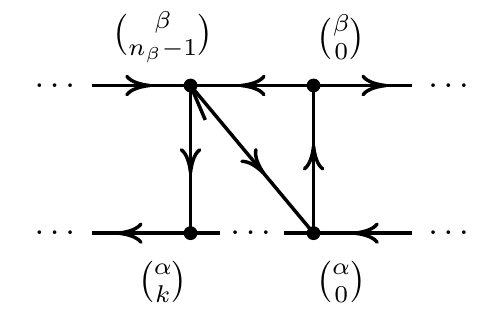}
		\caption*{Case (ix)}\label{case_ix}
	\end{subfigure}\\
\newline
	\begin{subfigure}{0.33\textwidth}\centering		
		%case (x)
		\includegraphics{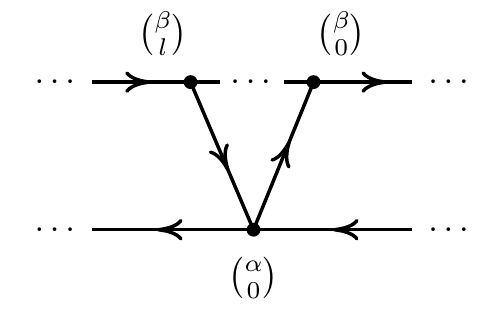}
		\caption*{Case (x)}\label{case_x}
	\end{subfigure}&
	
	\begin{subfigure}{0.33\textwidth}\centering		
		%case (xi)
		\includegraphics{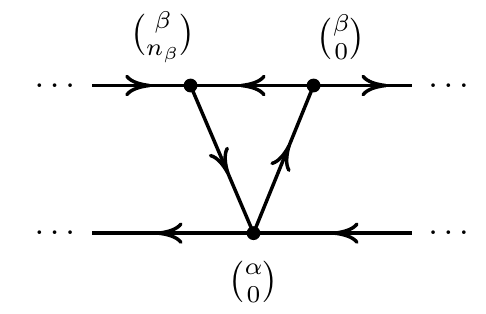}
		\caption*{Case (xi)}\label{case_xi}
	\end{subfigure}&
	
	\begin{subfigure}{0.33\textwidth}\centering
		%case (xii)
		\begin{tikzpicture}

		\end{tikzpicture}
	\end{subfigure}\\
	
	\end{tabular}
\end{table}
}

\begin{enumerate}[label = (\roman*)]\label{chord_list}
\item \emph{2 horizontal chords:} ${\al \choose n_\al -1}\rightarrow {\al \choose 0},\ {\be\choose 0}\rightarrow{\be\choose n_\be -1};$
\item \emph{1 horizontal chord:}  ${\al\choose 0}\rightarrow{\al \choose n_\al-1}$, $l\leq n_\be-2$ necessarily;
\item \emph{1 horizontal chord:} ${\al\choose n_\al -1}\rightarrow{\al \choose 0}$, $l\leq n_\be-2$ necessarily;
\item \emph{1 horizontal chord, 1 diagonal chord:} ${\al\choose n_\al -1}\rightarrow{\al \choose 0}$, ${\be\choose 0}\rightarrow{\al\choose n_\al -1}$, $l\leq n_\be-2$ necessarily;
\item \emph{1 horizontal chord, 1 diagonal chord:} ${\al\choose n_\al -1}\rightarrow{\al \choose 0}$, ${\al\choose 0}\rightarrow{\be\choose l}$, $l\leq n_\be-2$ necessarily;
\item \emph{1 horizontal chord:} ${\be\choose n_\be-1}\rightarrow{\be\choose 0}$, $k\leq n_\al-2$ necessarily;
\item \emph{1 horizontal chord:} ${\be\choose 0}\rightarrow{\be\choose n_\be-1}$, $k\leq n_\al-2$ necessarily;
\item \emph{1 horizontal chord, 1 diagonal chord:} ${\be\choose 0}\rightarrow{\be\choose n_\be-1}$, ${\al\choose k}\rightarrow{\be\choose 0}$, $k\leq n_\al-2$ necessarily;
\item \emph{1 horizontal chord, 1 diagonal chord:} ${\be\choose 0}\rightarrow{\be\choose n_\be-1}$, ${\be\choose n_\be -1}\rightarrow{\al\choose 0}$, $k\leq n_\al-2$ necessarily;

\item \emph{2 diagonal chords:} ${\be\choose l}\rightarrow{\al\choose 0}$, ${\al\choose 0}\rightarrow{\be\choose 0}$, $l\leq n_\be-2$ necessarily;
\item \emph{1 horizontal chord, 2 diagonal chords:} ${\be\choose n_\be-1}\rightarrow{\al\choose 0}$, ${\al\choose 0}\rightarrow{\be\choose 0}$.
\end{enumerate}

\noindent Direct check shows that for all of the cases above the statement of the lemma holds. Note that in the last two cases $k = n_\al$ and there are two non-trivial diagonal chords with arrows ${\be\choose l}\rightarrow{\al\choose 0}$ and ${\al\choose 0}\rightarrow{\be\choose 0}$. Cutting along these chords gives the following (see defn. below):
\begin{equation}
\cut_{{\be\choose l}\rightarrow{\al\choose 0}}C = \cut_{{\al\choose 0}\rightarrow{\be\choose 0}}C = \left[ {\al\choose 0}\rightarrow{\be\choose 0}\rightarrow\ldots\rightarrow{\be\choose l}\rightarrow{\al\choose 0} \right]
\end{equation}
\end{proof} 

Lemma \ref{simple_chords} allows us to define two cutting operations for $\al\bar\be$-cycle along its chord $b$. The first sends $C = a_1a_2\ldots a_k$ to $\cut_bC:= a_1\ldots a_{i-1}ba_{j+1}\ldots a_k$, the other sends it to $b\Rightcircle C := a_i\ldots a_j$. More generally, for a collection $(b_k,i_k,j_k)_{k\in \ZZ/l\ZZ}$ of $l$ nonintersecting chords of $C$ (that is $j_k<i_{k+1}\forall k\in \ZZ/l\ZZ$), $C$ can be cut along all chords in the obvious sense providing cycle denoted $\cut_{b_1b_2\ldots b_l} C$. In particular, for an empty collection of chords $\cut_\varnothing C = C$ by definition.

Let $U\in R\langle\langle Q\rangle\rangle$ and $p$ be a path in $Q$ identified with an homogeneous element of the path algebra. Denote $U[p]$ the coefficient of $p$ in expression of $U$ (this is well-defined after the choice of one-dimensional arrow spaces is specified). The following lemma follows directly from definitions:

\begin{lemma}\label{phiW}
Let $C$ be an $\al\bar\be$-cycle in $Q(\widetilde x)$ satisfying (\ref{exclusions},\ref{connect}). Then for any right-equivalence $\varphi$ and any potenial $W$ the following holds:

\begin{equation}\label{coeff}
\varphi(W)[C] = \sum_{b_1,b_2,\ldots b_l} W[\cut_{b_1b_2\ldots b_l}C]\prod_{r = 1}^l\varphi(b_r)[b_r\Rightcircle C],
\end{equation}
where the sum is over all collections $(b_1,b_2,\ldots b_k)$of nonintersecting chords of $C$.
\end{lemma}

\medskip
\noindent \emph{Proof of Proposition \ref{phi1}.}  For simplicity we can assume that the linear part of right-equivalence $\varphi$ is  trivial: $\varphi^{(1)} = \operatorname{id}$. The idea is to use Lemma \ref{phiW} to isolate terms of right-equivalence $\varphi$ that contribute to $\varphi(W)$, and then use the condition that right-equivalence $\varphi$ maps $W$ to primitive potential to extract relations for these terms. 

By the formula (\ref{coeff}) all nontrivial contributions to the coefficient of some $\al\bar\be$-cycle $C$ in $\varphi(W)$ come from cutting a cycle along a non-empty collection of nonintersecting chords. If $C$ is chordless then nothing except for $C$ itself contributes to the coefficient, otherwise the full list of possible configurations of chords in Lemma \ref{simple_chords} can be used for a case by case study.

\begin{figure}[h!]
	\begin{minipage}[b][][b]{.33\textwidth} %chord_i
		\centering
			\captionsetup{width = 1\textwidth}
		% simple example of mutation: initial quiver
		\includegraphics{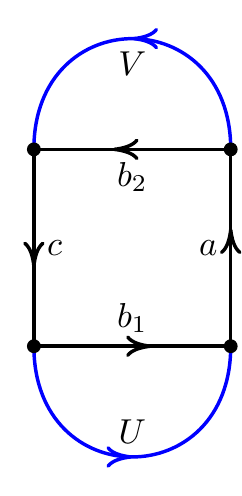}
		\captionof{figure}{Case (i)}
		\label{chord_i}
	\end{minipage}
	\begin{minipage}[b][][b]{.33\textwidth} %chord_iv
		\centering
		\captionsetup{width = 1.2\textwidth}
		\includegraphics{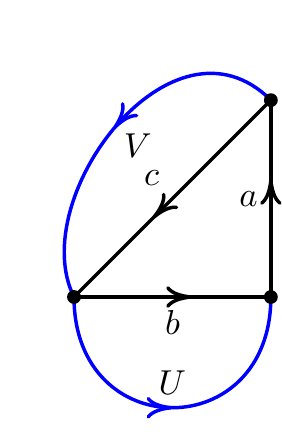}
		\captionof{figure}{Cases (iv,v,viii,ix)}
		\label{chord_iv}
	\end{minipage}
	% chord_x
	\begin{minipage}[b][][b]{.32\textwidth}
		\centering
		\captionsetup{width = 1\textwidth}
\includegraphics{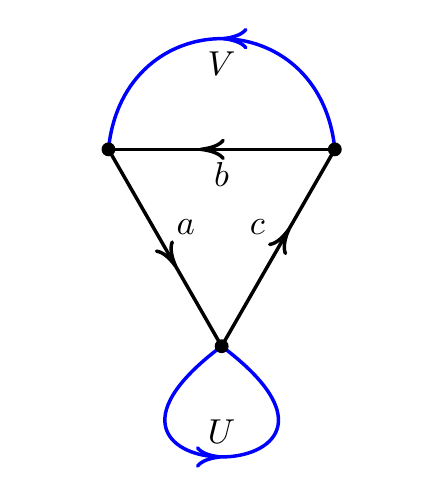}
		\captionof{figure}{Cases (x,xi)}
		\label{chord_x}
	\end{minipage}
\end{figure}

\medskip
\noindent\emph{Cases (ii,iii,vi,vii):} Since $C$ has a unique chord $b$ and $W[\cut_bC]=0$, the identity $\varphi(W)[C] = \varphi^{(1)}W[C]$ is immediate from (\ref{coeff}).

\medskip
\noindent\emph{Case (i):} The part of the potential of $Q(\widetilde x)$ restricted to $\al$- and $\be$-perches has the form:

\begin{equation}\label{W_i}
\begin{split}
\xi & \left[{\al\choose n_\al-1}\rightarrow{\al\choose 0}  \rightarrow{\be\choose 0}\rightarrow{\be\choose n_\be-1}\right]+\\
&\hspace{100pt}+\zeta\left[{\al\choose 0}\rightarrow{\be\choose 0}\rightarrow\ldots\rightarrow{\be\choose n_\be-1}\rightarrow{\al\choose n_\al-1}\right],
\end{split}
\end{equation}

\noindent with $\xi,\zeta\neq 0$. All other cycles in $W$ pass through the vertices decorated by different $\gamma\in\Gamma$ and can be omitted in the consideration because they do not affect $\varphi(W)[C]$. We depict the relevant part of the quiver in Figure \ref{chord_i}. Using notations from there (\ref{W_i}) becomes $\xi ab_1cb_2 + \zeta aUcV$.  Then by formula (\ref{coeff}) nontrivial contribution to $\varphi(W)[C]$ comes from $\cut_{b_1b_2}C$. Terms of $\phi^{(2)}$ that affect it are $b_1\mapsto b_1 + \mu_1U+\ldots$ and $b_2\mapsto b_2+\mu_2V+\ldots$. It follows that the following terms appear in $\varphi(W)$:

$$
\xi a(b_1 + \mu_1U)c(b_2 + \mu_2V)+\zeta aUcV = \xi ab_1cb_2+\xi\mu_1aUcb_2+\xi\mu_2ab_1cV+\xi\mu_1\mu_2aUcV+\zeta aUcV.
$$

\noindent Since $\varphi(W)$ is primitive we have to have zero coefficients in front of $aUcb_2$ and $ab_1cV$. Moreover, (\ref{exclusions}) implies that $b_2$ and $b_1$ are unique chords for $ab_1cV$ and $aUcb_2$ respectively. It follows that $\mu_1 = \mu_2 = 0$, otherwise $\varphi(W)[aUcb_1]$ or $\varphi(W)[ab_1cV]$ is nonzero. The proposition in case (i) is proven.
\medskip

\noindent\emph{Cases (iv,v,viii,ix):}. The relevant part of the quiver can be depicted as in Figure \ref{chord_iv}. The part of $W$ of our interest is $\xi abc+\zeta aUV$, terms in $\varphi$ are $b\mapsto b+\mu U+\ldots$ and $c\mapsto c+\nu V+\ldots$. Thus $\varphi(W)$ gets components:

$$
\xi a(b+\mu U)(c+\nu V)+\zeta aUV = \xi abc + \xi\mu aUc + \xi\nu abV +\xi\mu\nu aUV + \zeta aUV.
$$

\noindent As by assumption $\varphi(W)$ is primitive, coefficients $\varphi(W)[abV]$ and  $\varphi(W)[aUc]$ are zero. Since their only chords are $c$ and $b$ respectively, it follows that $\mu = \nu =0$ and the statement holds in these cases.
\medskip

\noindent\emph{Cases (x,xi):} The quiver spanned by $\al$- and $\be$-perches is redrawn in Figure \ref{chord_x}. The part of potential for this subquiver is $\xi abc+\zeta aVcU$. Terms in $\varphi$ contributing in $\varphi(W)[C]$ are:

$$
\centering
\begin{aligned}
&a\mapsto a+\lambda Ua+\ldots\\
&b\mapsto b+\mu V+\ldots \hspace{15pt}\text{--- only in case (x)}\\
&c\mapsto c+\nu cU+\ldots
\end{aligned}$$

\noindent Hence the contribution to potential is:

\begin{equation}\label{W_x}
\begin{split}
\xi(a+\lambda Ua)(b+\mu V)(c+\nu cU) + \zeta(a+\lambda Ua)&V(c+\nu cU)U = \\
=\xi abc + \xi\lambda Uabc + \xi\mu aVc + U& \xi\nu abcU +\\
+ \xi\lambda\mu UaVc +\xi\mu\nu aVcU&+\xi\mu\lambda UabcU+\xi\lambda\mu\nu UaVcU +\\
+\zeta aVcU&+\zeta\lambda UaVcU +\zeta\nu aVcUU +\zeta\lambda\nu UaVcUU.
\end{split}
\end{equation}

\noindent Since cycle $aVc$ has the unique chord $b$, the condition $\varphi(W)[aVc]=0$ implies $\mu = 0$. 
It is enough to conclude that $W[aVcU] = \varphi(W)[aVcU]$ and the proposition for cases (x,xi) follows.

\medskip
The proof of Proposition \ref{phi1} is complete.

\bigskip
Thus, to describe the space of primitive potentials modulo the action of right-equivalences it is enough to consider linear equivalences $\varphi = \varphi^{(1)}$. Since all arrow spaces in $Q(x)$ are one-dimensional, $\varphi^{(1)}$ acts diagonally.

Any quiver $Q(\widetilde x)$, satisfying (\ref{exclusions},\ref{connect}) gives rise to a $2$-dimensional complex $\mathcal{C}(\widetilde x)$ defined as follows: $1$-skeleton is identified with $Q(\widetilde x)$; for every $\al\bar\be$-cycle there is a $2$-cell attached to $Q(\widetilde x)$ along arrows of the cycle. Note that in our construction orientation of $2$-cells is consistent with orientation of arrows of the quiver under differential in the CW-complex:

$$\dd C = \sum_{i=1}^k a_i,$$

\noindent for any $\al\bar\be$-cycle $C = a_1a_2\ldots a_k$.

Here is our first main result:
\begin{theorem}
For the quiver $Q(\widetilde x)$ satisfying (\ref{exclusions},\ref{connect}) the space of primitive potentials modulo the action of right-equivalences preserving this space is isomorphic to the second cohomology group of $\mathcal{C}(\widetilde x)$ with coefficients in \emph{$\kk^\times$}.
\end{theorem}

\begin{proof} Any primitive potential $W$ defines a 2-cocycle in $C^2(\mathcal{C}(\widetilde x),\kk^\times)$ in a obvious way.
By Proposition \ref{phi1} it is enough to consider action of diagonal right-equivalences $\varphi = \varphi^{(1)}$ that multiply every arrow $a$ by some nonzero number $\xi_a$. Then it is immediate from the construction of $\mathcal{C}(\widetilde x)$ that the action of such $\varphi$ on $W$ is equivalent to adding the differential of $1$-cochain $\xi_a$ to the $2$-cochain defined by $W$.
\end{proof}

It is worth mentioning that topologically $CW$-complexes under consideration are homotopy equivalent to direct product $\Gamma\times S^1$. Thus the second cohomology group is $$H^2(\mathcal{C}(\widetilde x),\kk^\times) = (\kk^\times)^{\operatorname{rk} H^1(\Gamma,\ZZ)}$$

\subsection{Mutation of primitive potentials on $Q(\widetilde x)$.}

In this subsection we prove two results. The first states that if $x,y\in\MM_\Gamma$ differ by a single braid relation (see \ref{braid_rels}), then CW-complexes $\mathcal{C}(\widetilde x)$ and $\mathcal{C}(\widetilde y)$ are homotopy equivalent. This allows us to identify second cohomology groups of these spaces. Our second result makes connection between cohomology classes represented by primitive potentials on quivers related by mutations (see Proposition \ref{braid_mutations} and Lemmas \ref{lemma_aba},\ref{lemma_aa}). It turns out that this cohomology class is almost preserved up to a certain change of signs. To describe this relation precisely we introduce \emph{twisted cohomology class} associated to any primitive potential $W_x$ on $Q(\widetilde x)$.

Recall that $W_x$ defines an element $h({W_x})\in H^2(\mathcal{C}(\widetilde x),\kk^\times)$ by taking sum of all $\al\bar\be$-cycles with coefficients from the expression for $W_x$. The twist is defined by inserting signs in this expression in a way that will be described momentarily. Note that numbers of $\al\bar\be$- and $\be\bar\al$-cycles in $Q(\widetilde x)$ are equal because they alternate along the subquiver spanned by $\al$- and $\be$- perches. Let $k_{\al\be}$ be this number, so that there is a total of $2k_{\al\be}$ cycles sitting between these two perches. To twist $h(W_x)$ we choose one such cycle for every unordered pair $\{\al,\be\}$ satisfying $C_{\al\be}=-1$. The twisted cohomology class represented by $W_x$ is obtained from $h({W_x})$ by multiplying the coefficient of every chosen cycle by $(-1)^{k_{\al\be}}$ (one cycle per pair of neighboring perches). It is easy to see that the new cohomology class does not depend on the choice of cycles between neighboring perches. We denote the twisted cohomology class associated to $W_x$ by $h'({W_x})$.

\begin{theorem}\label{mut_pot}
Let $x,y\in\MM_\Gamma$ be two words such that $y$ can be obtained from $x$ using a single relation from the list \ref{braid_rels}. Then by Proposition \ref{braid_mutations} either $Q(\widetilde x)$ and $Q(\widetilde y)$ are identical, or $Q(\widetilde y) = \mu_{\al \choose r}Q(\widetilde x)$. In the latter case following statements hold:

\begin{enumerate}[label=(\roman*)]
 \item CW-complexes $\mathcal{C}(\widetilde x)$ and $\mathcal{C}(\widetilde y)$ are homotopy equivalent;
 \item For any primitive potential $W_x$ on $Q(\widetilde x)$,  the potential $\mu_{\al \choose r}W_x$ on $Q(\widetilde y)$ is primitive;
 \item Twisted cohomology classes $h'({W_x})$ and $h'\left(\mu_{\al \choose r}W_x\right)$ represent the same second cohomology class under identification from ~(i).
 \end{enumerate}
\end{theorem}

Recall from the argument in the proof of Proposition \ref{braid_mutations} that $Q(\widetilde y) = \mu_{\al \choose r}Q(\widetilde x)$, if $x$ and $y$  are related by one of the two relations: $s_\al  s_\be s_\al \leftrightarrow s_\be s_\al s_\be$ or $s_\al s_{\bar\al}\leftrightarrow s_{\bar\al}s_\al$. These cases are considered below.

We need the following technical statement that describes in explicit terms the process of cancelling a $2$-cycle in a quiver with potential.

\begin{figure}
	\centering
	\begin{minipage}{.5\textwidth}
		\centering
		%Step 1 picture (WARNING:mixed node names)
			\includegraphics{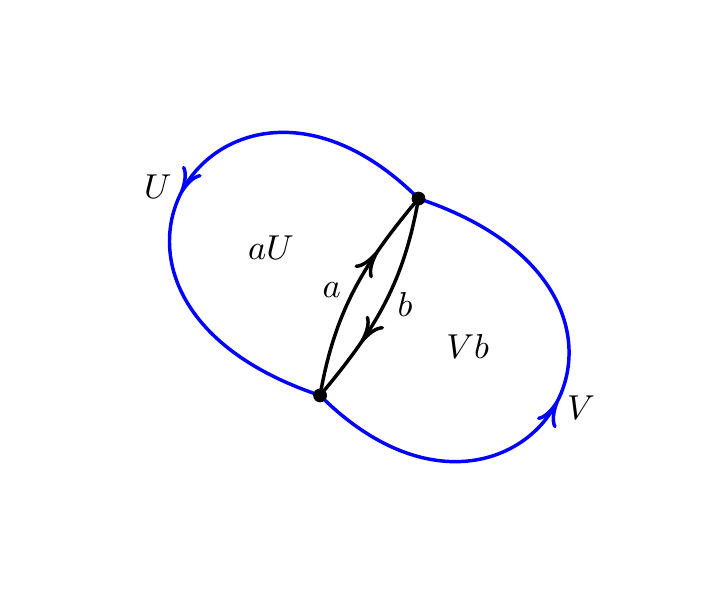}
	\end{minipage}%
	\begin{minipage}{.5\textwidth}
	\centering
		\includegraphics{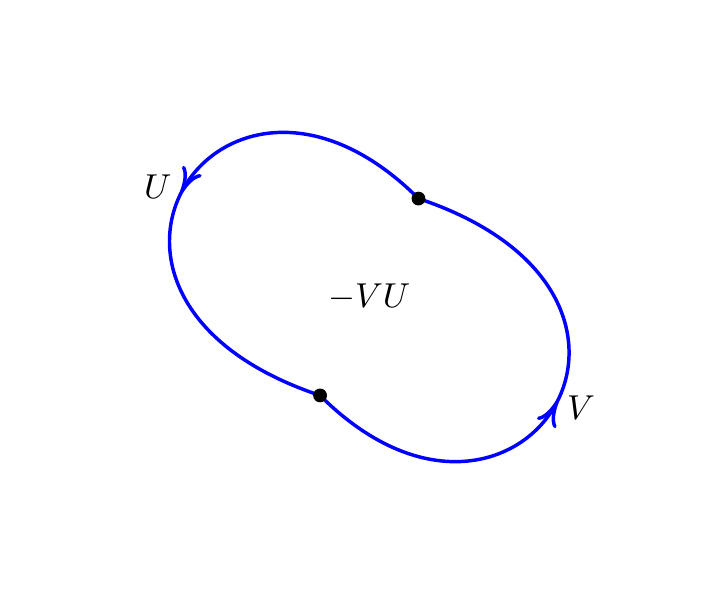}		
	\end{minipage}
	\captionof{figure}{Deleting $2$-cycle}
	\label{delete}
\end{figure}

\begin{lemma} \label{red}
Let $(Q,W)$ be a quiver with potential where $W$ has form $ab+a U+Vb+W_0$, where 

\begin{enumerate}[label=(\roman*)]
\item $ab$ is the product of two arrows forming a $2$-cycle;
\item $U$ and $V$ are \underline{any} elements composable with $a$ and $b$ respectively;
\item $U,V,W_0$ contain no letters $a$ or $b$.
\end{enumerate}
	
\noindent Then $(Q,W)$ is right-equivalent to $(Q',W')\oplus(Q_{\operatorname{triv}},W_{\operatorname{triv}}=ab)$, where $Q'$ is obtained from $Q$ by deleting arrows $a$ and $b$, and $W'=-VU+W_0$ (see Fig. \ref{delete}).
\end{lemma}

\begin{proof}
 Apply right-equivalence $\al\mapsto\al- V,\ \be\mapsto\be- U$, then the potential takes form:
 $$W\mapsto(a-V)(b- U)+(a- V)U+V(b- U)+W_0=ab- VU+W_0,$$ 
by assumptions term $U,V,W_0$ includes no arrows $a,b$, so $(Q,W)$ can be decomposed into a direct sum. The statement follows.
\end{proof}

\begin{figure}
	\centering
	\begin{minipage}{.5\textwidth}
	\captionsetup{width=0.8\textwidth}
		\centering
		\captionsetup{width=.95\textwidth}
		\includegraphics{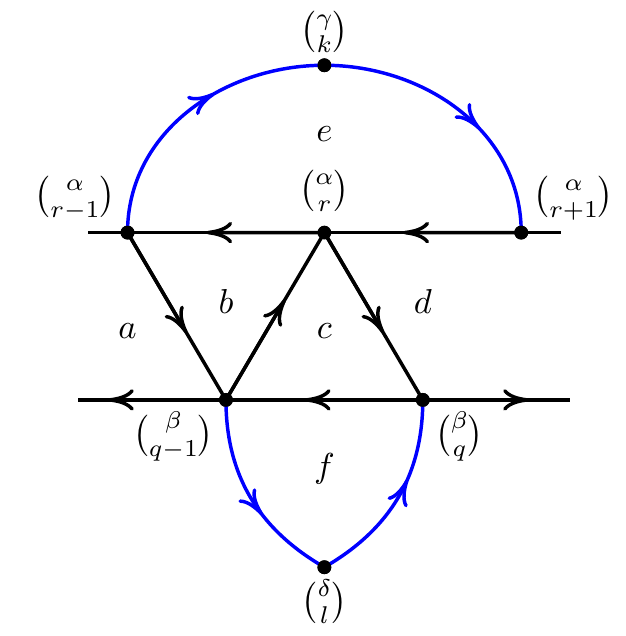}
		\captionof{figure}{Case 1: $(Q(\widetilde x), W_x)$}
		\label{step1}
	\end{minipage}%
	\begin{minipage}{.5\textwidth}
		\centering
		\captionsetup{width=1\textwidth}
		\includegraphics{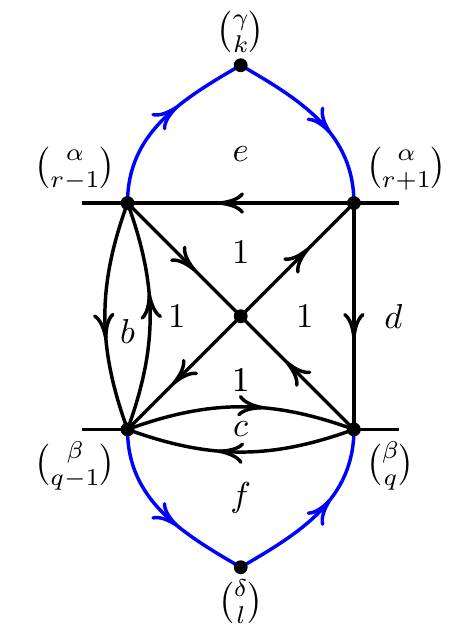}
		\captionof{figure}{Case 1: $(\widetilde{\mu}_{\al\choose r}Q(\widetilde x),\widetilde{\mu}_{\al\choose r}W_x)$}
		\label{step2}
	\end{minipage}
\end{figure}

\begin{figure}
	\centering
	\begin{minipage}{.5\textwidth}
	\captionsetup{width=1\textwidth}
		\centering
	%Step3
	\includegraphics{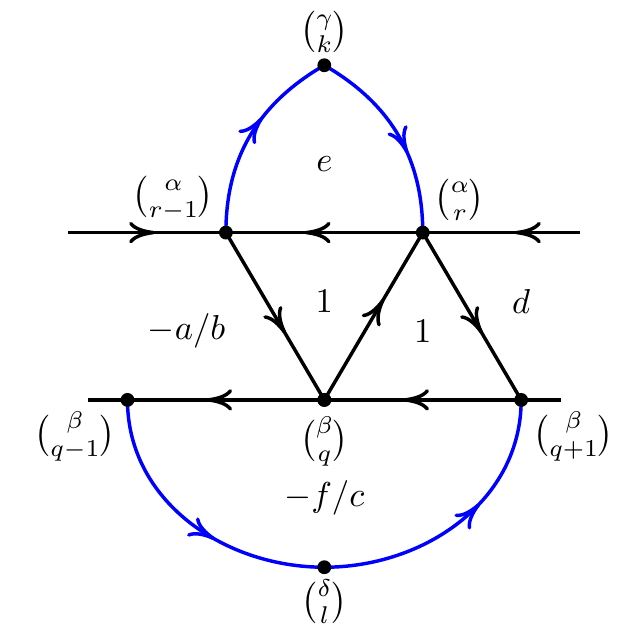}
		\captionof{figure}{Case 1: $(Q(\widetilde y),\mu_{\al\choose r}W_x)$}
		\label{step3}
	\end{minipage}%
	\begin{minipage}{.5\textwidth}
	\captionsetup{width=.95\textwidth}
		\centering
	%Step4
		\includegraphics{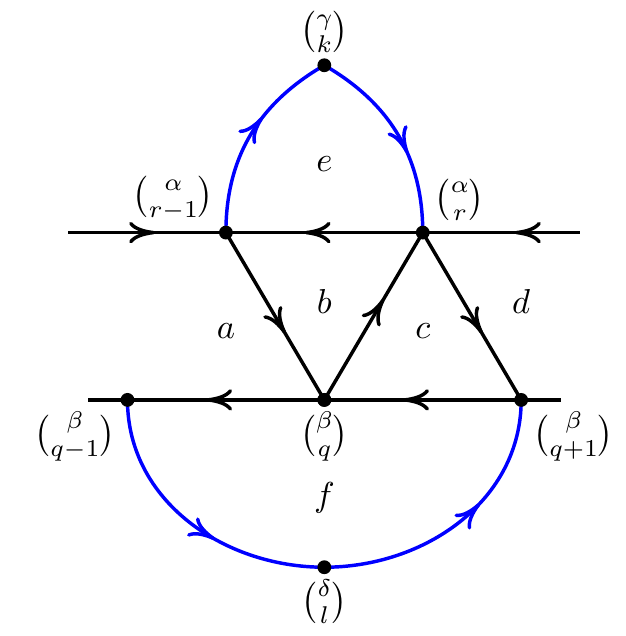}
		\captionof{figure}{Case 1: $(Q(\widetilde y),\mu_{\al\choose r}W_x)$ after the right-equivalence}
		\label{step4}
	\end{minipage}
\end{figure}

\noindent\emph{Proof of Theorem \ref{mut_pot}.} Suppose that word $y$ is obtained from $x$ by replacing a fragment $s_\al^{(r)} s_\be^{(q)} s_\al^{(r+1)}$ with $s_\be^{(q)} s_\al^{(r)} s_\be^{(q+1)}$. Recall from Lemma \ref{lemma_aba} that there are exactly  four arrows adjacent to the mutated vertex $\al\choose r$: ${\al \choose r}\longrightarrow {\al \choose r-1}$, ${\al \choose r+1}\longrightarrow {\al \choose r}$, ${\be \choose q-1}\longrightarrow {\al \choose r}$ and ${\al \choose r}\longrightarrow {\be \choose q}$. It is easy to see that for our statements it is enough to consider only the part of  $Q(\widetilde x)$ that consists of vertices connected with ${\al \choose r}$, and only cycles in the potential $W$ passing through ${\al \choose r}$ or through its neighbors. Indeed, the rest of the picture is intact under the mutation. %Example of the corresponding part of the quiver and its mutation is shown in Figures \ref{step1}-\ref{step4}.

\medskip
Since all $2$-cycles appearing after premutaion appear with nonzero coefficients (by definition of primitive potential), Lemma \ref{red} justifies that all of them can be deleted after reduction.
First two parts of the theorem follow from this description and calculation shown in Figures \ref{step1}-\ref{step4} and \ref{step1v2}-\ref{step2v2}.

Note that in our figures blue arrows arrows represent sums of paths going from $\al$- or $\be$-perch to any other neighboring perch. Thus, arrow ${\be \choose q}\longrightarrow {\be \choose q-1}$ is included in $\be\bar\gamma$-cycle $\forall\gamma\in\Gamma$ with $C_{\al\gamma} = -1$, and similarly for arrows ${\al \choose r}\longrightarrow {\al \choose r-1}$, ${\al \choose r+1}\longrightarrow {\al \choose r}$. However, $\forall\gamma\neq\be$ with $C_{\al\gamma} = -1$ both arrows belong to the same $\al\bar\gamma$-cycle. Also we have to relabel vertices to match vertices of $\mu_{\al\choose r}Q(\widetilde x)$ with vertices of $Q(\widetilde y)$.

\medskip

\begin{figure}
	\centering
	\begin{minipage}{.5\textwidth}
		\centering
		%Step 1v2 picture (WARNING:mixed node names)
		\includegraphics{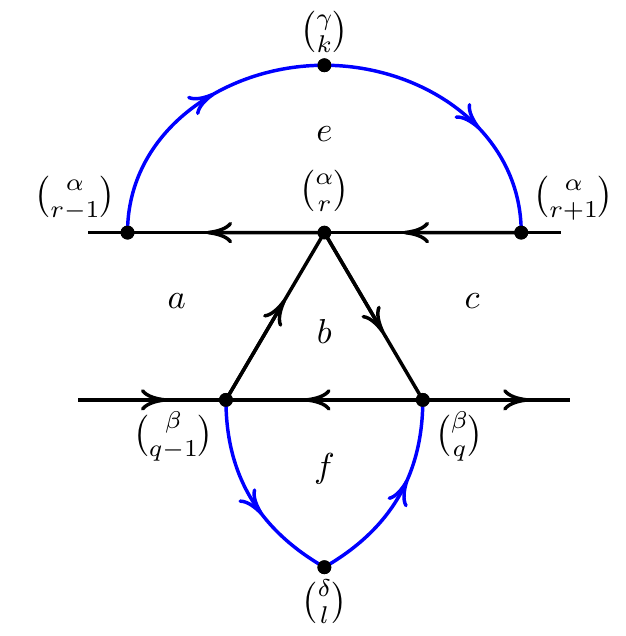}
		\captionof{figure}{Case 2: ($Q(\widetilde x),W_x$)}
		\label{step1v2}
	\end{minipage}%
	\begin{minipage}{.5\textwidth}
		\centering
		\includegraphics{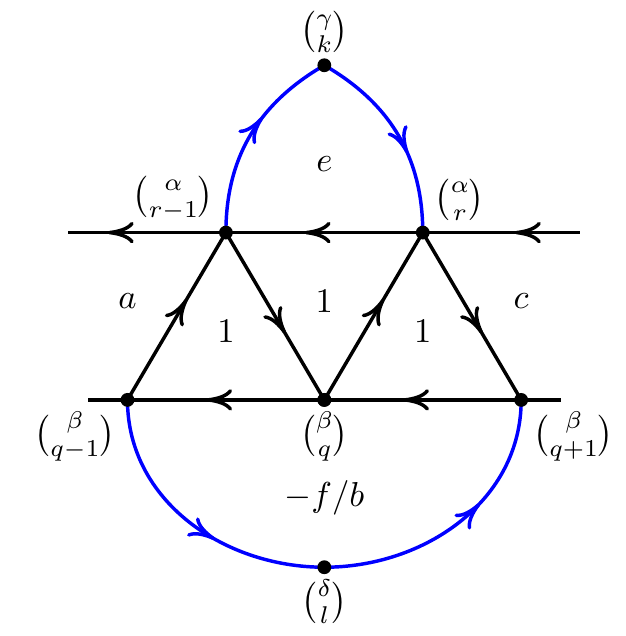}
		\captionof{figure}{Case 2: $(Q(\widetilde y),\mu_{\al\choose r}W_x)$}
		\label{step2v2}
	\end{minipage}
\end{figure}

It remains to check the identity of twisted cohomology classes. For that we take all cycles in $Q(\widetilde x)$ involved in the mutation with arbitrary coefficients (depicted in the interior of cycles in Fig. \ref{step1},\ref{step1v2}). And we compute the change of these coefficients under $\mu_{\al\choose r}$ using Lemma \ref{red} (see Fig. \ref{step3},\ref{step2v2}).

In fact, precise form of $\al\bar\be$- and $\be\bar\al$-cycles involved depends on other letters in words $x$ and $y$. Explicitly, vertices ${\al \choose r-1}$ and ${\be \choose q-1}$ are joined by an arrow in $Q(\widetilde x)$, if in the expression for $x$ some $s_\be$ or $s_{\bar{\al}}$ stands closer on the left to $s_\al^{(r)}$  than any letters $s_\al$ or $s_{\bar{\be}}$ (Fig. \ref{step1}); otherwise this arrow is absent (Fig. \ref{step1v2}). Similarly there are $2$ such possibilities for vertices ${\be \choose q}$ and ${\al \choose r+1}$ resulting in a total of $4$ cases. However, since mutations act as an involution it is enough to consider only half of these cases: either only one arrow is present or, both arrows are absent. Other cases are obtained on the other side of the involution.

\medskip
Now consider the case when $Q(\widetilde x)$ contains an arrow ${\al\choose r-1}\rightarrow{\be\choose q-1}$ and no arrows between ${\al\choose r+1}$ and ${\be\choose q}$ (Figure \ref{step1}). Then the effect of the mutation on the initial potential is shown in the picture \ref{step3}. To pass to the potential in Figure \ref{step4} we apply right-equivalence:

$$
\begin{cases}
\left[{\be\choose q-1}\rightarrow{\al\choose r-1}\right]\mapsto b\left[{\be\choose q-1}\rightarrow{\al\choose r-1}\right]\\
\left[{\be\choose q}\rightarrow{\be\choose q-1}\right]\mapsto -\left[{\be\choose q}\rightarrow{\be\choose q-1}\right]\\
\left[{\be\choose q+1}\rightarrow{\be\choose q}\right]\mapsto c\left[{\be\choose q+1}\rightarrow{\be\choose q}\right]\\
\end{cases}
$$

\noindent Note that after contracting the arrow ${\al\choose r+1}\rightarrow{\al\choose r}$ in Figure \ref{step1} and the arrow ${\be\choose q}\rightarrow{\be\choose q-1}$ in Figure \ref{step4} resulting CW-complexes and potentials become identical. Since the number $k_{\al\be}$ of $\al\bar\be$-cycles is unchanged this implies required equality of twisted cohomology classes.

\medskip
Consider the case when both arrows between ${\al\choose r-1}$ and ${\be\choose q-1}$ and between $\be\choose q$ and $\al\choose r+1$ are absent (Figure \ref{step1v2}). The effect of the mutation on potential is shown in Figure \ref{step2v2}. To compare potential we first apply the right-equivalence $\left[{\be\choose q+1}\rightarrow{\be\choose q}\right]\mapsto -b\left[{\be\choose q+1}\rightarrow{\be\choose q}\right]$.

Now we need to identify second cohomology groups of the two CW-complexes in concrete terms. This can be done by applying the following homotopy equivalences: first contract the arrow ${\al\choose r+1}\rightarrow{\al\choose r}$ in Figure \ref{step1v2} and then contract  ${\be\choose q}\rightarrow{\be\choose q-1}$ in Figure \ref{step2v2}. This way we can embed the former CW complex in the latter. There is an inconsistency though: the coefficient of the cycle $\left[{\be\choose q+1}\rightarrow{\be\choose q}\rightarrow{\al\choose r}\rightarrow{\be\choose q+1}\right]$ is $-b$ instead of $b$. But this is exactly where the twisting comes into play --- the number of $\al\bar\be$-cycles has increased by one. Thus, the equality of the twisted cohomology classes follows, and the theorem is proven for the first type of relations $s_\al s_\be s_\al \leftrightarrow s_\be s_\al s_\be$.

Assume that $y$ is obtained from $x$ by a single relation $s_\al^{(r)} s^{(r+1)}_{\bar\al}\longleftrightarrow s^{(r)}_{\bar\al}s^{(r+1)}_\al$. Recall from Lemma \ref{lemma_aa} that for the quiver $Q(\widetilde x)$ arrows adjacent to the mutated vertex $\al\choose r$ are the following: ${\al\choose r}\rightarrow{\al\choose r-1}$, ${\al\choose r}\rightarrow{\al\choose r+1}$ and one arrow ${\al\choose r}\rightarrow {\be\choose q}$ for every $\be\in\Gamma$ with $C_{\al\be} = -1$.

As in the argument above, there can be different configurations of $\al\bar\be$- and $\be\bar\al$-cycles next to $\al\choose r$ depending on the form of the word $x$. Namely, there can be or there can be no arrows between $\al\choose r-1$ and $\be\choose q$ and between $\al \choose r+1$ and $\be\choose q$. It is enough to consider two out of four possibilities: the other two are obtained on the mutated side $Q(\widetilde y)$ via remark that mutations act as involutions.

\begin{figure}
	\begin{minipage}{.5\textwidth}
		\centering 
%mutation of a\bar{a} Step 1
\includegraphics{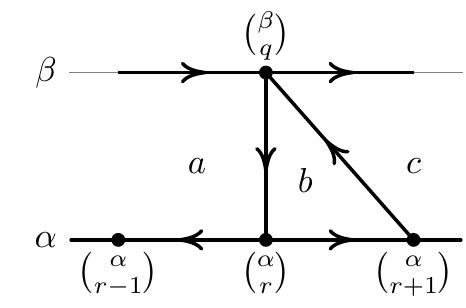}
		\captionof{figure}{Case 1: $(Q(\widetilde x),W_x)$}
		\label{aa_1}
\end{minipage}
\begin{minipage}{.49\textwidth}
		\centering 
		\captionsetup{width=.95\textwidth}
% mutation of a\bar{a} Step 2
\includegraphics{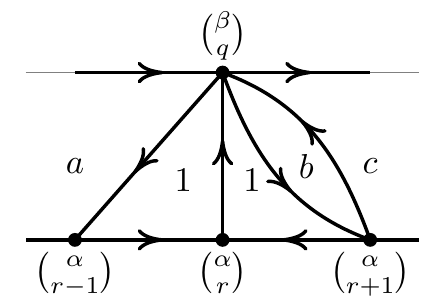}
		\captionof{figure}{Case 1: $(\widetilde{\mu}_{\al\choose r}Q(\widetilde x), \widetilde{\mu}_{\al\choose r}W_x)$}
		\label{aa_2}
\end{minipage}		
\end{figure}
\begin{figure}
% mutation of a\bar{a} Step 3
\begin{minipage}{.5\textwidth}
\captionsetup{width=.95\textwidth}
		\centering 
\includegraphics{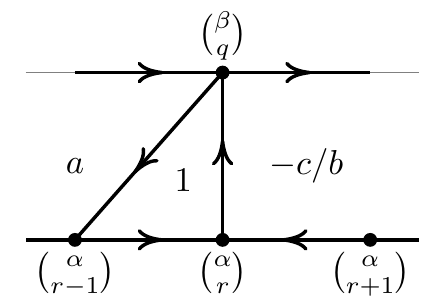}
		\captionof{figure}{Case 1: $(Q(\widetilde y), \mu_{\al\choose r}W_x)$}
		\label{aa_3}
\end{minipage}
% mutation of a\bar{a} Step 4
\begin{minipage}{.49\textwidth}
\captionsetup{width=.95\textwidth}
		\centering 
\includegraphics{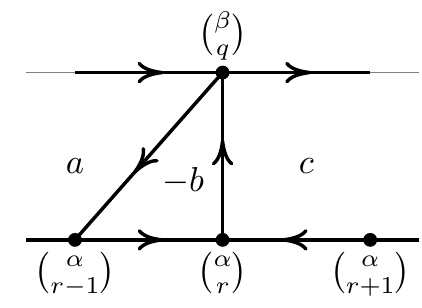}
		\captionof{figure}{Case 1: $(Q(\widetilde y), \mu_{\al\choose r}W_x)$ after the right-equivalence}
		\label{aa_4}
\end{minipage}
\end{figure}

First case we consider is when the arrow from $\al\choose r-1$ to $\be\choose q$ is absent and there is an arrow from $\al\choose r+1$ to $\be\choose q$. Corresponding sequence of mutations is shown in Figures \ref{aa_1}-\ref{aa_4}. Between last to steps right-equivalence sending $\left[ {\al\choose r}\rightarrow{\be\choose q}\right]\mapsto -b \left[ {\al\choose r}\rightarrow{\be\choose q}\right]$ is applied. Observe that the cycle $\left[{\al\choose r-1}\rightarrow{\al\choose r}\rightarrow{\be\choose q}\rightarrow{\al\choose r-1}\right]$ in Figure \ref{aa_4} is taken with coefficient $-b$ and in Figure \ref{aa_1} there is a cycle appearing with coefficient $b$. We will explain how to deal with this discrepancy after considering the second case.

\begin{figure}[h!]
	\begin{minipage}{.5\textwidth}
		\centering 
%mutation of a\bar{a} Step 1
\includegraphics{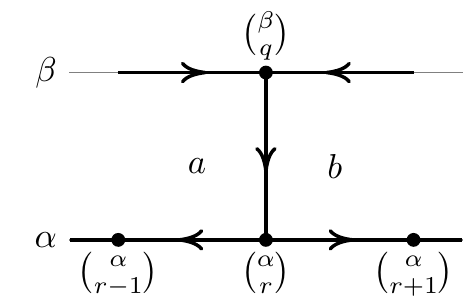}
		\captionof{figure}{Case 2: $(Q(\widetilde x),W_x)$}
		\label{aa_1v2}
\end{minipage}
\begin{minipage}{.49\textwidth}
		\centering 
% mutation of a\bar{a} Step 2v2
\includegraphics{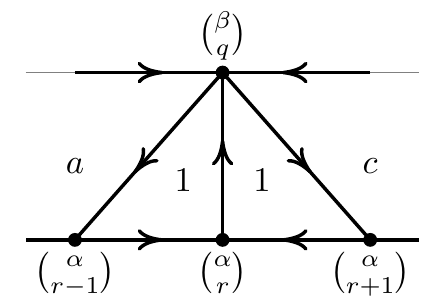}
		\captionof{figure}{Case 2: $(Q(\widetilde y),\mu_{\al\choose r}W_x)$}
		\label{aa_2v2}
\end{minipage}		
\end{figure}

Suppose there are no arrows from  $\al\choose r-1$ to $\be\choose q$ and from $\al\choose r+1$ to $\be\choose q$. The computation is straightforward in this case and is shown in Figures \ref{aa_1v2},\ref{aa_2v2}. 

To compare cohomology groups of CW-complexes $\mathcal{C}(\widetilde x)$ and $\mathcal{C}(\widetilde x)$ note that we can just contract arrows between ${\al\choose r}$ and ${\al\choose r-1}$ and between ${\al\choose r}$ and ${\al\choose r+1}$ on both sides. Under this identification for every $\be\in\Gamma$ from the first case the number $k_{\al\be}$ of $\al\bar\be$-cycles in $Q(\widetilde x)$ and  $Q(\widetilde y)$ is the same, but as noted above the coefficient of one of the cycles is negated. For every $\beta\in\Gamma$ from the second case the number of $\al\bar\be$-cycles is $1$ more in $Q(\widetilde y)$, but coefficients of all $\al\bar\be$- and $\be\bar\al$-cycles agree under our identification. Both of this discrepancies are fixed by applying right-equivalence that changes the sign of arrow ${\al\choose r-1}\rightarrow{\al\choose r}$ in $Q(\widetilde y)$. This concludes the full proof of Theorem \ref{mut_pot}.

\medskip
It is easy to see in the proof of the theorem that conditions (\ref{exclusions},\ref{connect}) can be relaxed. For example, as noted above (\ref{connect}) is only necessary for connectedness of quivers and can be dropped. Absence of (\ref{exclusions}) can produce quivers with $2$-cycles or loops, so one has to be more careful with definitions of quivers with potentials and their mutations. However, the construction allows to deal at least with $2$-cycles in cases when there is a mutation-equivalent quiver for which  (\ref{exclusions},\ref{connect}) hold.

\bibliographystyle{alphaurl}
\bibliography{references}

\end{document}